\documentclass[12pt, authoryear]{article}
\usepackage{amsfonts, graphicx,color,multirow, amsthm,amsmath,natbib}
\usepackage{url}
\usepackage{multirow}
\usepackage{subfigure}

\usepackage{enumerate}
\usepackage{algorithmic}
\usepackage{algorithm}

\makeatletter
\theoremstyle{plain}
\newtheorem{thm}{\protect\theoremname}

\usepackage{verbatim}

\makeatletter
\usepackage{amsthm}
\usepackage{amsmath}
\usepackage{setspace}
\usepackage{amssymb}

\makeatletter
  \theoremstyle{definition}
  \newtheorem{condition}{Condition}
  \theoremstyle{plain}
  
  \theoremstyle{plain}

\theoremstyle{plain}
  \theoremstyle{remark}
  
  \newcommand{\R}{I\!\!R}
  \newcommand{\mC}{\mathcal{C}}
   \newcommand{\mD}{\mathcal{D}}
\providecommand{\\}{\\}
\DeclareMathOperator{\Corr}{Corr}
\DeclareMathOperator{\E}{\mathbb{E}}

\newcommand{\Etal}{{ et al.}}

\newcommand{\bx}{\mbox{\bf x}}

\newcommand{\be}{\mbox{\bf e}}

\newcommand{\bX}{\mbox{\bf X}}
\newcommand{\bsX}{\mbox{\scriptsize \bf X}}

\newcommand{\bZ}{\mbox{\bf Z}}

\newcommand{\bbeta}{\mbox{\boldmath $\beta$}}
\newcommand{\bsbeta}{\mbox{\scriptsize \boldmath $\beta$}}

\newcommand{\bSigma}{\mbox{\boldmath $\Sigma$}}

\newcommand{\FDR}{\mathrm{FDR}}

\newcommand{\Var}{\mbox{Var}}
\newcommand{\Cov}{\mathrm{Cov}}

\newcommand{\cov}{\mathrm{Cov}}

\newcommand{\argmin}{\mathrm{argmin}}

\def\cov{\mbox{Cov}}

\newcommand{\lp}{\left(}
\newcommand{\rp}{\right)}

\makeatother

\providecommand{\theoremname}{Theorem}

\begin{document}

\title{\bf Conditional Sure Independence Screening \thanks{Emre Barut is graduate student (Email: abarut@princeton.edu), Jianqing Fan is Frederick L. Moore'18 professor (Email: jqfan@princeton.edu), Department of Operations Research \& Financial Engineering, Princeton University, Princeton, NJ 08544, USA. Anneleen Verhasselt is assistant professor, Department of Mathematics and Computer Science,  University of Antwerp, Belgium.  The paper was initiated while Anneleen Verhasselt was a visiting postdoctoral fellow at Princeton University.  This research was partly supported by NSF Grant DMS-1206464,  NIH Grants R01-GM072611 and R01-GM100474, and FWO Travel Grant V422811N.
}
}

\author{Emre Barut, Jianqing Fan \\ Princeton University \and Anneleen Verhasselt\\ University of Antwerp}
\date{}

\maketitle

\vspace*{-0.5 in}
\onehalfspacing
\begin{abstract}
Independence screening is a powerful method for variable selection for `Big Data' when the number of variables is massive. Commonly used independence screening methods are based on marginal correlations or variations of it.  In many applications, researchers often have some prior knowledge that a certain set of variables is related to the response.  In such a situation, a natural assessment on the relative importance of the other predictors is the conditional contributions of the individual predictors in presence of the known set of variables.  This results in conditional sure independence screening (CSIS). Conditioning helps for reducing the false positive and the false negative rates in the variable selection process.  In this paper, we propose and study CSIS in the context of generalized linear models.  For ultrahigh-dimensional statistical problems, we give conditions under which sure screening is possible and derive an upper bound on the number of selected variables.  We also spell out the situation under which CSIS yields model selection consistency.  Moreover, we provide two data-driven methods to select the thresholding parameter of conditional screening. The utility of the procedure is illustrated by simulation studies and analysis of two real data sets.
\vfill

\end{abstract}

\vfill

\newpage
\doublespacing
\section{INTRODUCTION}

Statisticians are nowadays frequently confronted with massive data sets from various frontiers of scientific research.  Fields such as genomics, neuroscience, finance and earth sciences have different concerns on their subject matters, but nevertheless share a common theme:  They rely heavily on extracting useful information from massive data and the number of covariates $p$ can be huge in comparison with the sample size $n$. In such a situation, the parameters are identifiable only when the number of the predictors that are relevant to the response is small, namely, the vector of regression coefficients is sparse.  This sparsity assumption has a nice interpretation that only a limited number of variables have a prediction power on the response. To explore the sparsity, variable selection techniques are needed.

Over the last ten years, there has been many exciting developments in statistics and machine learning on variable selection techniques for ultrahigh dimensional feature space.  They can basically be classified into two classes:  penalized likelihood and screening.   Penalized likelihood techniques are well known in statistics: Bridge regression (\citealt{bridge}), Lasso (\citealt{Lasso}), SCAD or other folded concave regularization methods (\citealt{FanLi01, FL11, ZZ12}), and Dantzig selector (\citealt{DS,BRT09}), among others.  These techniques select variables and estimate parameters simultaneously by solving a high-dimensional optimization problem.  See \citet{ESL} and \cite{BV11} for an overview of the field.  Despite the fact that various efficient algorithms have been proposed (\citealt{LassoDual,OPT2,Lasso2,FL11}), statisticians and machine learners still face huge computational challenges when the number of variables is in tens of thousands of dimensions or higher. This is particularly the case as we are entering the era of ``Big Data'' in which both sample size and dimensionality are large.

With this background, \cite{SIS} propose a two-scale approach, called iterative sure independence screening (ISIS), which screens and selects variables iteratively.  The approach is further developed by \citet{SISML} in the context of generalized linear models. Theoretical properties of sure independence screening for generalized linear models have been thoroughly studied by \cite{SISGLM}.  Other marginal screening methods include tilting methods (\citealt{HTX09}), generalized correlation screening (\citealt{HM09}), nonparametric screening (\citealt{AddSIS}), and robust rank correlation based screening (\citealt{RobustSIS}), among others.  The merits of screening include expediences in distributed computation and implementation. By ranking marginal utility such as marginal correlation with the response, variables with weak marginal utilities are screened out by a simple thresholding.

The simple marginal screening faces a number of challenges.  As pointed out in \cite{SIS}, it can screen out those hidden signature variables: those who have a big impact on response but are weakly correlated with the response.  It can have large false positives too, namely recruiting those variables who have strong marginal utilities but are conditionally independent with the response given other variables.  \cite{SIS} and \cite{SISML} use a residual based approach to circumvent the problem but the idea of conditional screening has never been formally developed.

Conditional marginal screening is a natural extension of simple independent screening. In many applications, researchers know from previous investigations that certain variables $\bX_{\mC}$ are responsible for the outcomes. This knowledge should be taken into account when applying a variable selection technique in order not to remove these predictors from the model and to improve the selection process. Conditional screening recruits additional variables to strengthen the prediction power of $\bX_{\mC}$, via ranking conditional marginal utility of each variable in presence of $\bX_{\mC}$. In absence of such a prior knowledge, one can take those variables that survive the screening and selection as in \cite{SIS}.

Conditional screening has several advantages.  First of all, it makes it possible to recover the hidden significant variables.  This can be seen by considering the following linear regression model
\begin{equation} \label{eq1}
   Y=\bX^T \bbeta^\star +\varepsilon, \qquad E \bX \varepsilon = 0,
\end{equation}
with $\bbeta^\star=(\beta^\star_1,\ldots,\beta^\star_p)^T$.
The marginal covariance  between $X_j$ and $Y$ is given by
\[
\Cov(X_j,Y)=\Cov(X_j,\bX \bbeta)=\be_j^T \bSigma \bbeta^\star,
\]
where $\be_j \in \R^p$ is equal to 0, except for its $j$th element which equals to 1. This shows that the marginal covariance between $X_j$ and $Y$ is zero if $\beta^\star_j= -\sum_{k\neq j} \beta^\star_{k} \sigma_{kj}$, where $\sigma_{kj}$ is the $(k, j)$ element of $\bSigma = \Var(\bX)$, with $\bX=(X_1,\ldots,X_p)^T$.  Yet, $\beta^\star_j$ can be far away from zero.  In other words, under the conditions listed above, $X_j$ is a hidden signature variable.
To demonstrate that, let us consider the case in which $p=2000$, with true regression coefficients $\bbeta^\star= (3,3,3,3,3,-7.5,0, \cdots,0)^T$,  and all variables follow the standard normal distribution with equal correlation 0.5, and $\varepsilon$ follows the standard normal distribution.  By design, $X_6$ is a hidden signature variable, which is marginally uncorrelated with the response $Y$.  Based on a random sample of size $100$ from the model, we fit marginal regression and obtain the marginal estimates $\{\hat{\beta}_j^M\}_{j=1}^p$.  The magnitudes of these estimates are summarized by their averages over three groups:  indices 1 to 5 (denoted by $\bbeta_{1:5}^M$), 6 and indices 7 to 2000.
Clearly, the magnitude on the first group should be the largest, followed by the third group.  Figure~\ref{Fig1a} depicts the distributions of those marginal magnitudes based on 10000 simulations.  Clearly variable $X_6$ can not be selected by marginal screening.

\begin{figure}[ht]
\centering
\subfigure[Distribution of $\left| \hat{\beta}^M \right|$]{
\includegraphics[width=0.45\textwidth]{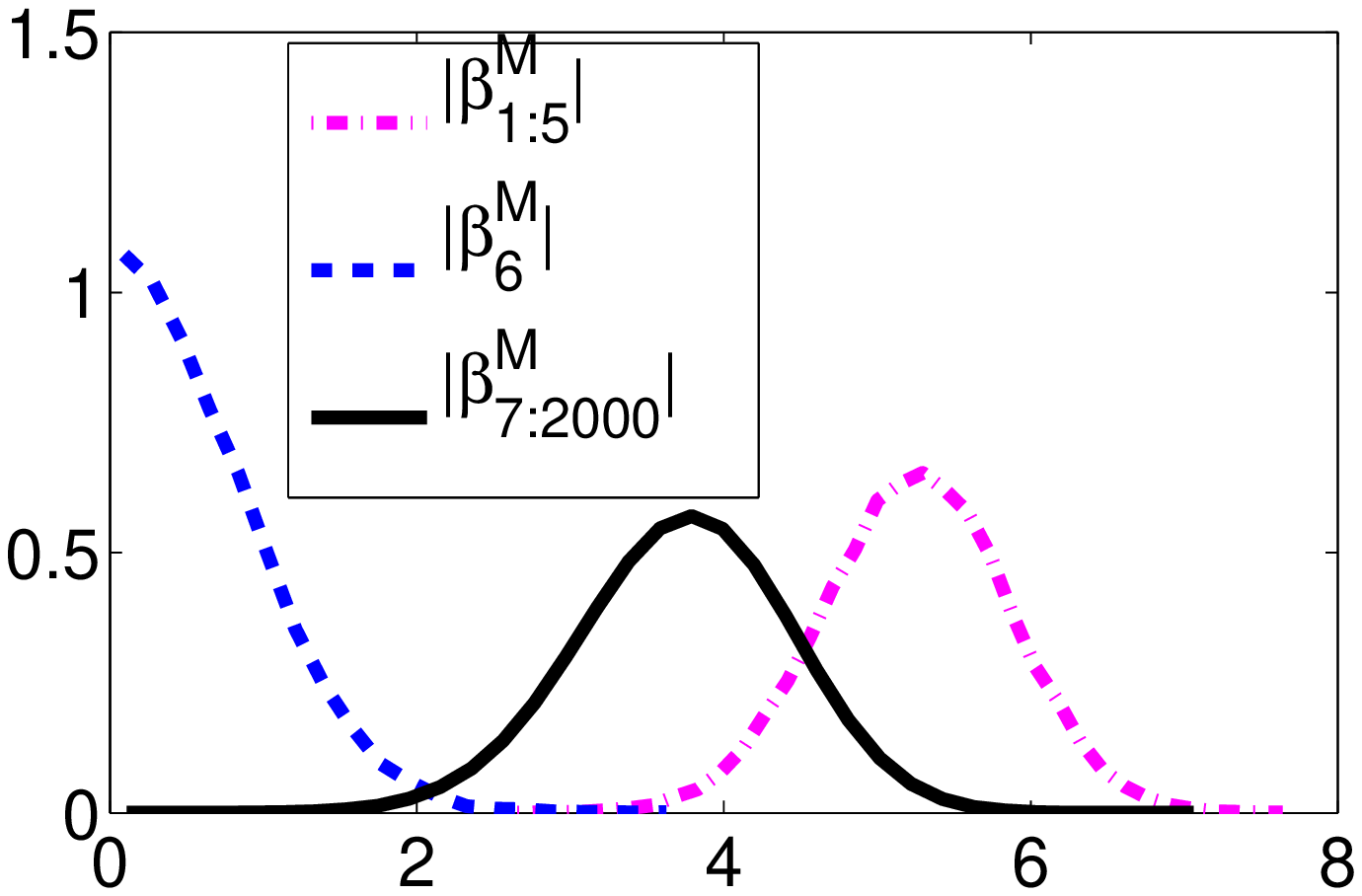}\label{Fig1a}
}
\subfigure[Distribution of $\left| \hat{\beta}_{\mC}^{M} \right|$]{
\includegraphics[width=0.45\textwidth]{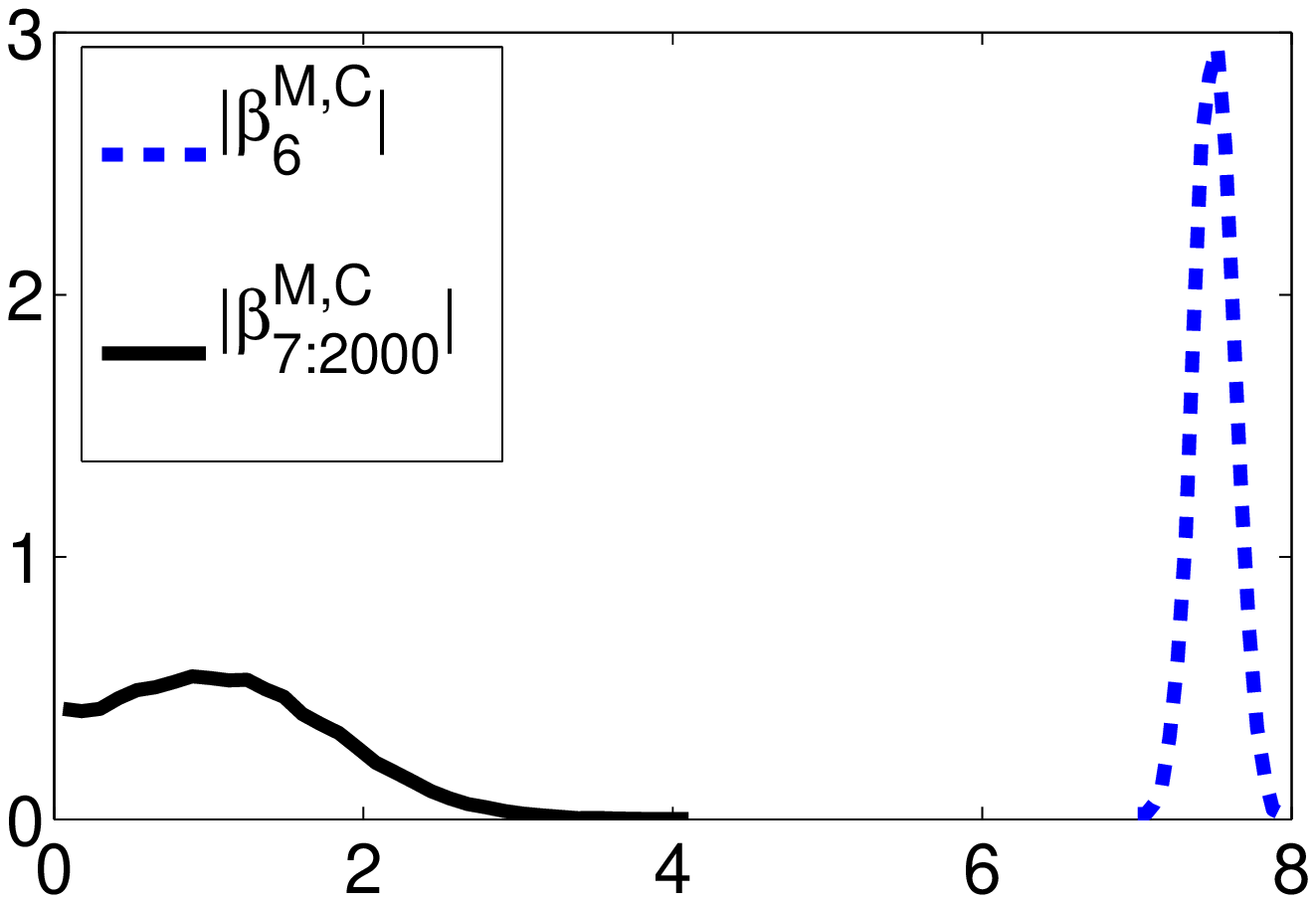}\label{Fig1b}
}\\
\subfigure[Distribution of $\left| \hat{\beta}_{\mC}^{M} \right|$]{
\includegraphics[width=0.45\textwidth]{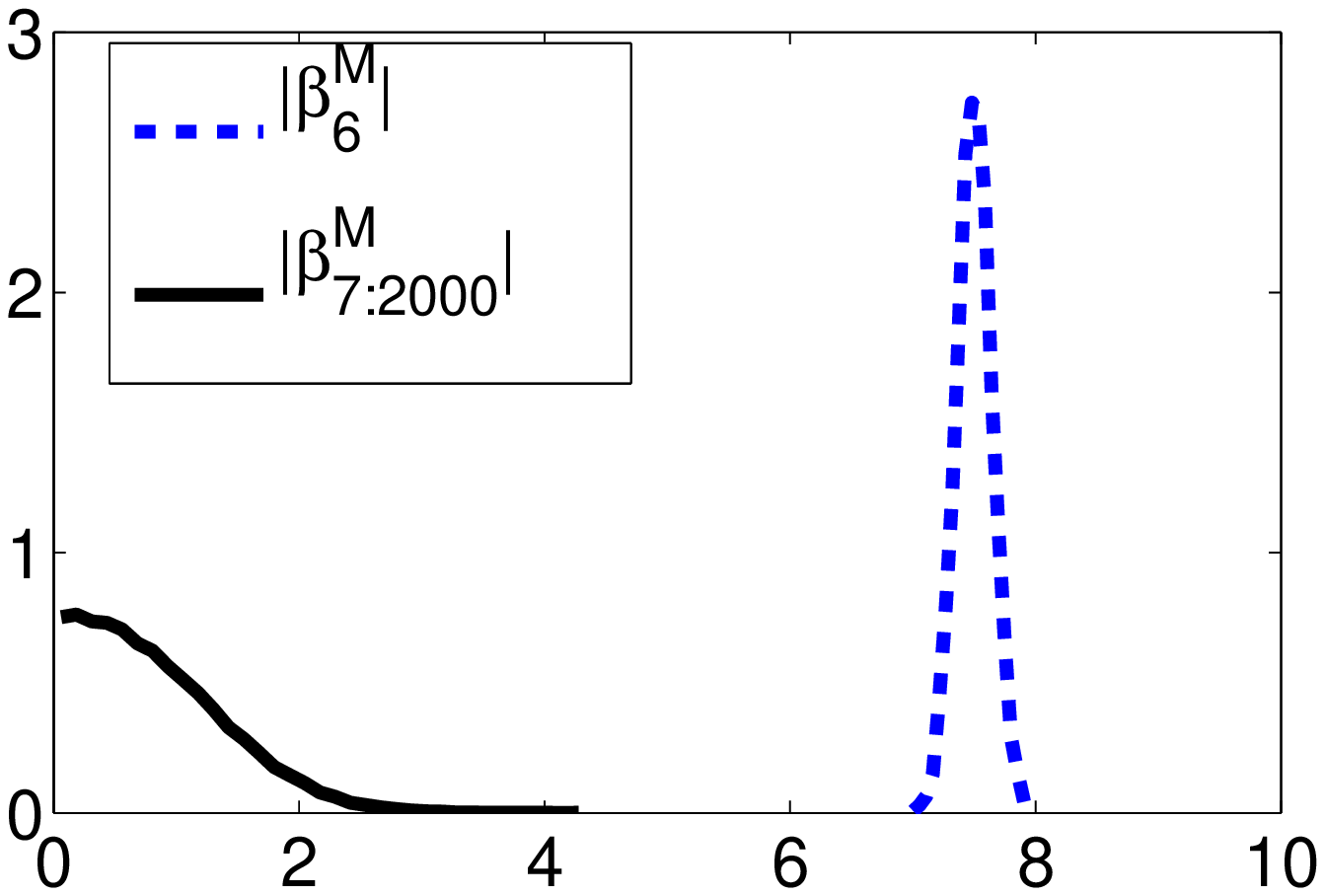}\label{Fig1c}
}
\subfigure[Distribution of $\left| \hat{\beta}_{\mC}^{M} \right|$]{
\includegraphics[width=0.45\textwidth]{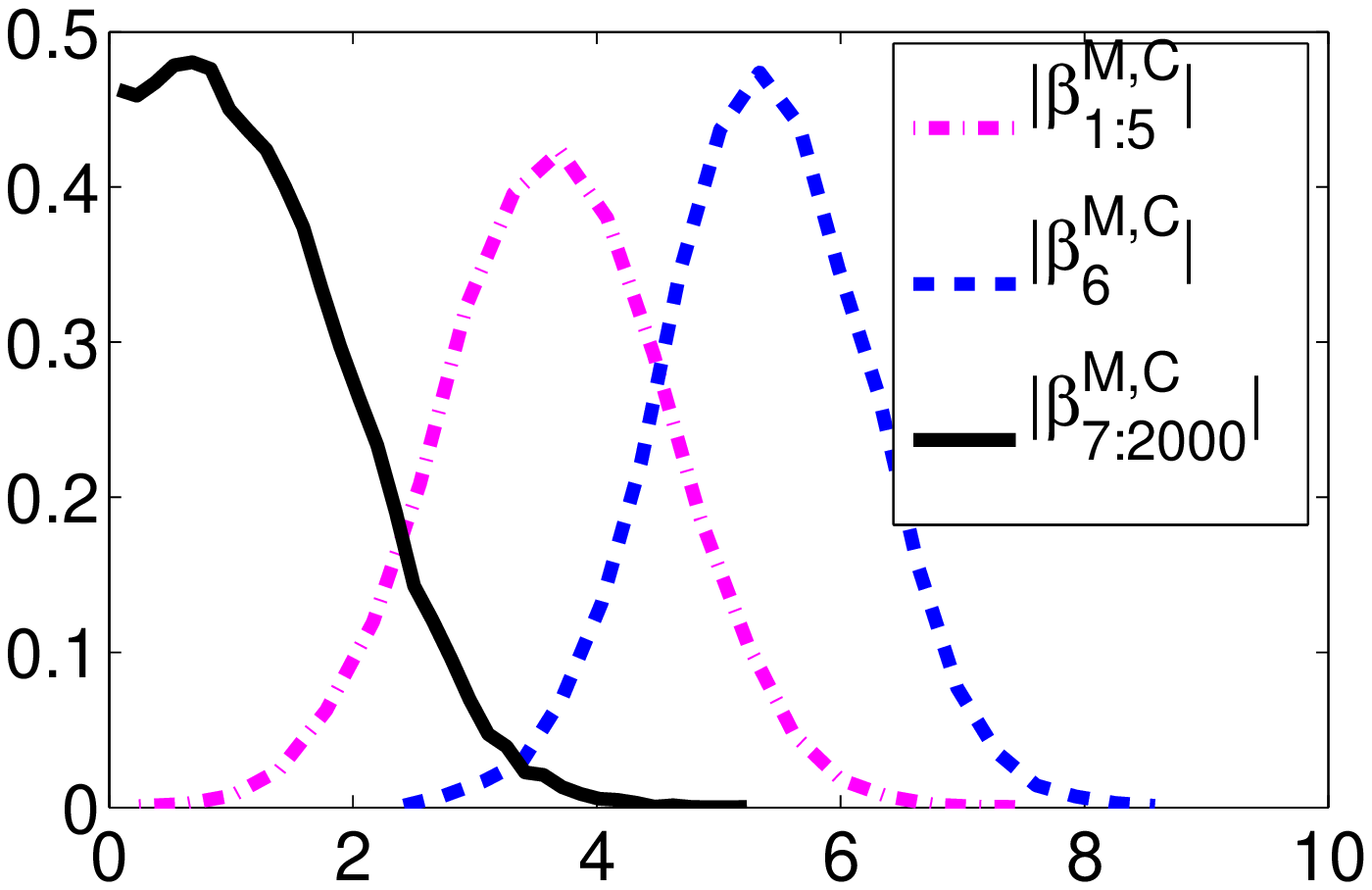}\label{Fig1d}
}
\caption{Benefits of conditioning against false negatives. Upper left panel: the distributions of the averages of magnitudes $|\hat{\beta}_j^{M}|$ of marginal regression coefficients over three groups of variables 1:5, 6, 7:2000. Upper right panel:  the distributions of the averages of the magnitude $|\hat{\beta}^{M}_{\mC j}|$ of conditional marginal regression coefficients over two groups of variables:  6 and 7:2000. Lower left panel: the distributions of the magnitudes $|\hat{\beta}^{M}_{\mC j}|$ of conditional marginal regression when the conditioned set includes inactive variables. Lower right panel: the distributions of the averages of the magnitude $|\hat{\beta}^{M}_{\mC j}|$ of conditional marginal regression coefficients given five randomly chosen variables.}
\label{Fig1}
\end{figure}

Adapting the conditional screening approach gives a very different  result. Conditioning upon the first five variables, conditional correlation between $X_6$ and $Y$ has a large magnitude. With the same simulated data as in the above example, the regression coefficient $\hat{\beta}_{\mC j}^{M}$ of $X_j$ in the joint model with the first five variables is computed.  This measures the conditional contribution of variable $X_j$ in presence of the first five variables.  Again, the magnitudes $\{|\hat{\beta}_{\mC j}^{M}|\}_{j=6}^{2000}$ are summarized into two values:  $|\hat{\beta}_{\mC 6}^{M}|$ and the average of $\{|\hat{\beta}_{\mC j}^{M}|\}_{j=7}^{2000}$.  The distributions of those over 10000 simulations are also depicted in Figure~\ref{Fig1b}.  Clearly, the variable $X_6$ has higher marginal contributions than others.  That is, conditioning helps recruiting the hidden signature variable. Furthermore, conditioning is fairly robust to extra elements. To demonstrate that, we have repeated the previous experiment with conditioning on five more randomly chosen features. The distribution of the magnitudes are given in Figure~\ref{Fig1c}. It is seen that the important hidden variable again has a large magnitude.

% new stuff
The benefits of conditioning are observed even if the conditioned variables are not in the active set. To demonstrate that, the regression coefficient $\hat{\beta}_{\mC j}^{M}$ of $X_j$ has been computed while conditioning on five randomly chosen inactive variables. That is, contribution of variable $X_j$ is calculated in the presence of these five randomly chosen inactive variables. The magnitudes of $\{|\hat{\beta}_{\mC j}^{M}|\}_{j=1}^{2000}$ are summarized in three groups: the average of the first five important variables, i.e. $\{|\hat{\beta}_{\mC j}^{M}|\}_{j=1}^{5}$, $|\hat{\beta}_{\mC 6}^{M}|$ and the average of $\{|\hat{\beta}_{\mC j}^{M}|\}_{j=7}^{2000}$. The distributions for these variables over 10000 simulations are given in Figure~\ref{Fig1d}.  It is observed that the magnitude of the hidden signature variable increases significantly and hence it will surely  not be missed during the screening. In other words, conditioning can help to recruit the important variables, even when the conditional set is not ideally chosen.
%check this again

\begin{figure}[ht]
\centering
\subfigure[Distribution of $\left| \hat{\beta}^M \right|$]{
\includegraphics[width=0.45\textwidth]{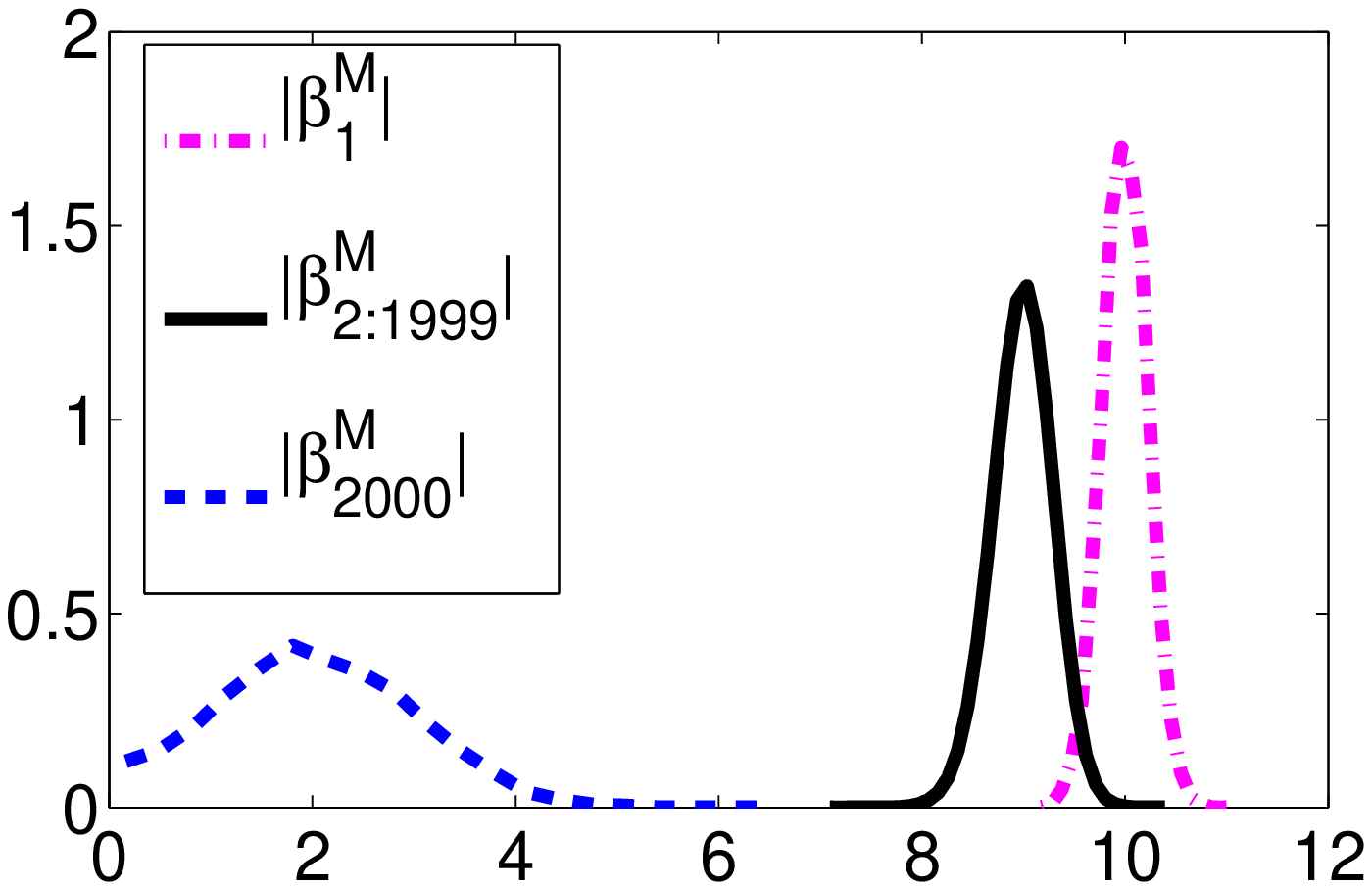}\label{Fig2a}
}
\subfigure[Distribution of $\left| \hat{\beta}_{\mC}^{M} \right|$]{
\includegraphics[width=0.45\textwidth]{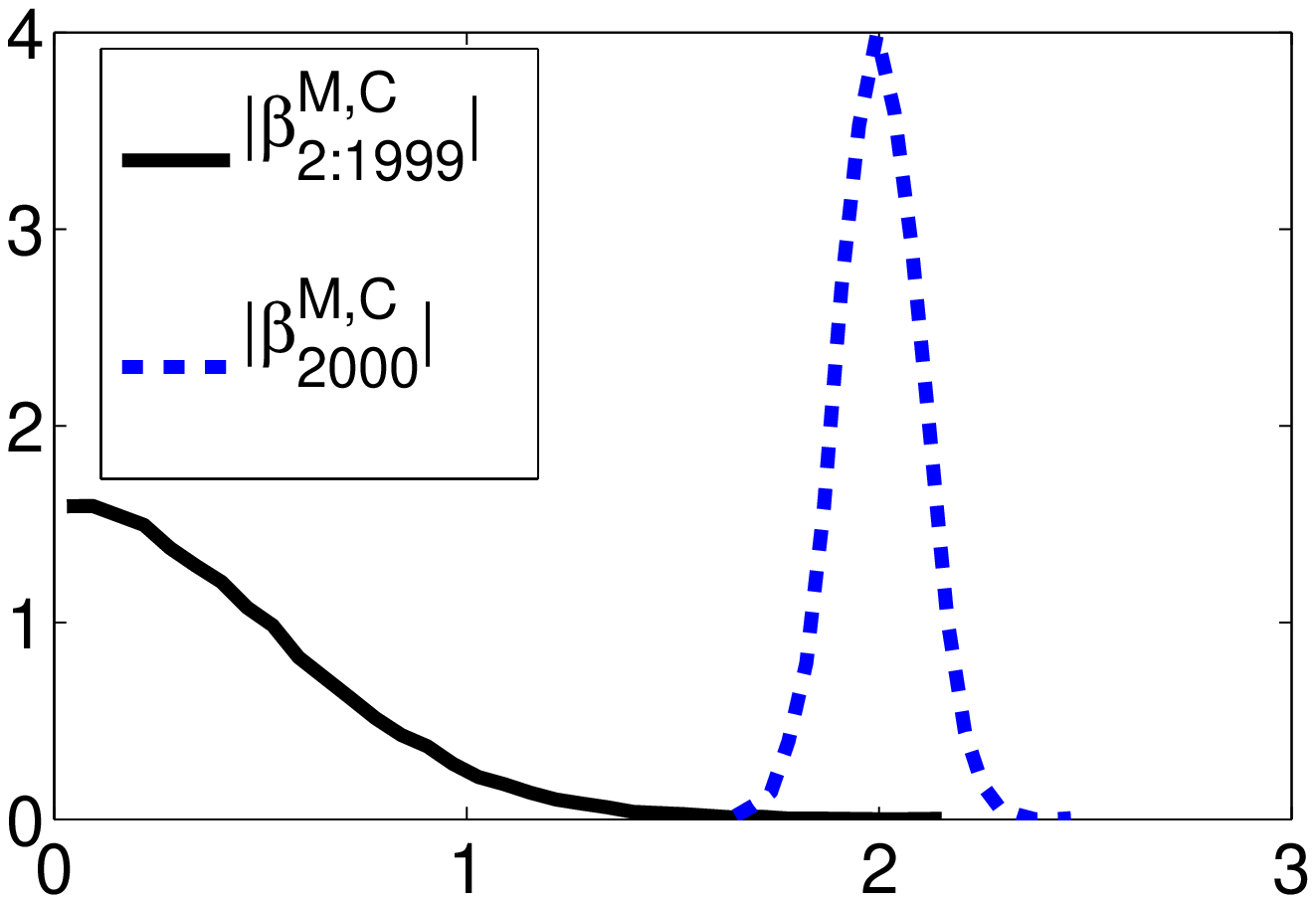}\label{Fig2b}
}\\
\subfigure[Distribution of $\left| \hat{\beta}_{\mC}^{M} \right|$]{
\includegraphics[width=0.45\textwidth]{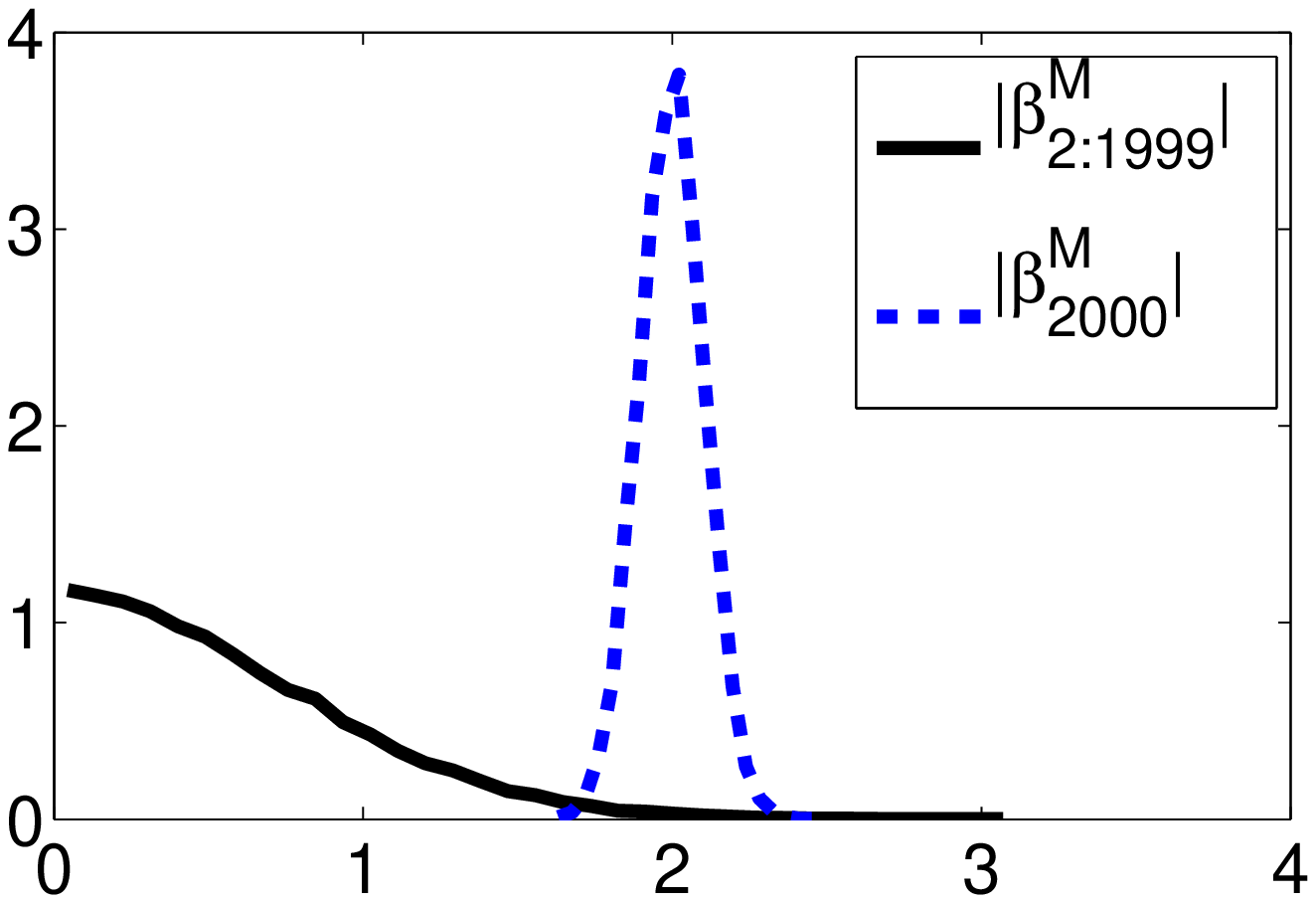}\label{Fig2c}
}
\subfigure[Distribution of $\left| \hat{\beta}_{\mC}^{M} \right|$]{
\includegraphics[width=0.45\textwidth]{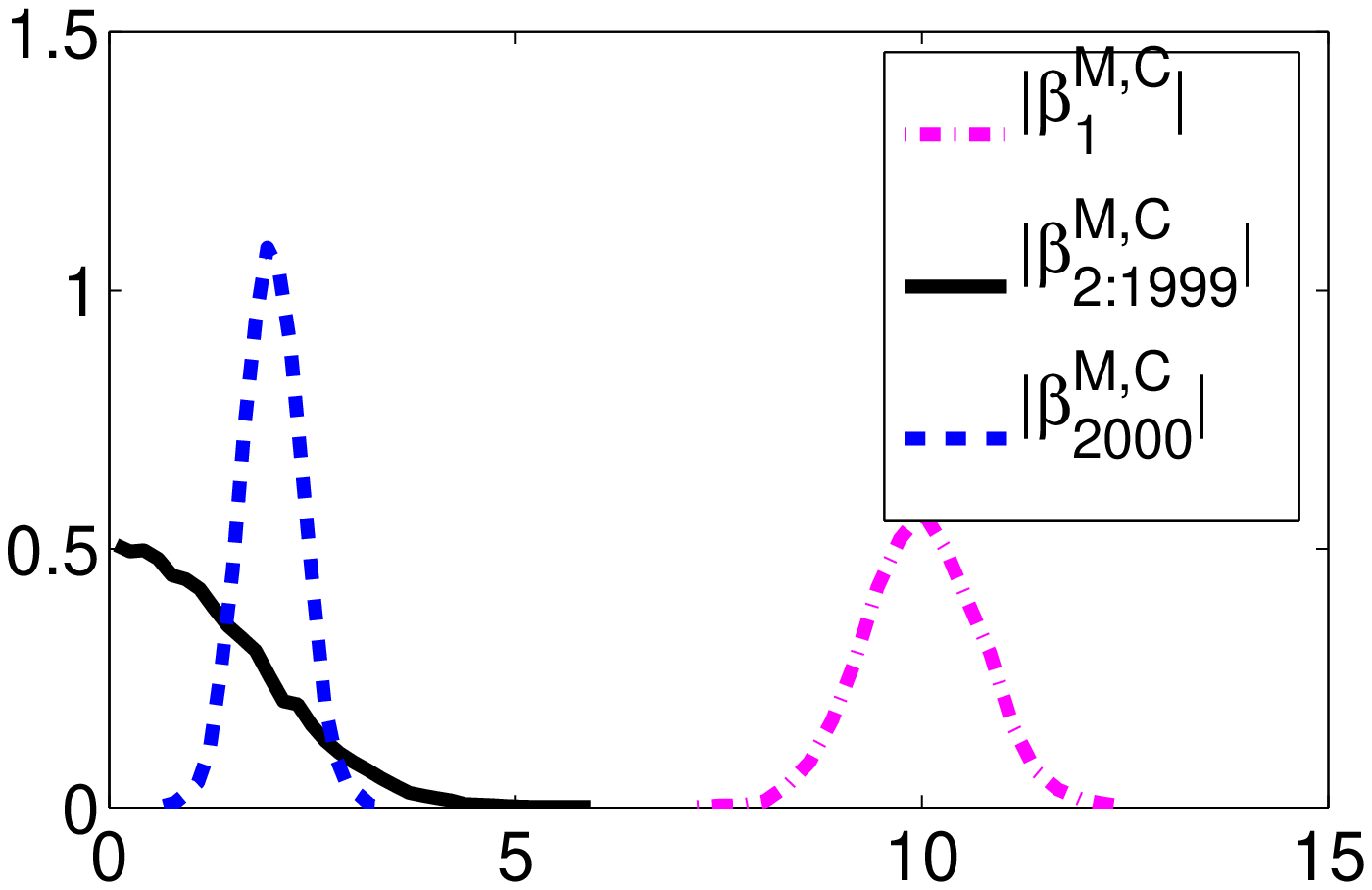}\label{Fig2d}
}
\caption{Benefits of conditioning against false positives.  Upper left panel:
the distributions of the magnitude $|\hat{\beta}_j^{M}|$ of marginal regression coefficients over three groups of variables 1, 2:1999 and 2000.  Upper right panel:  the distributions of the magnitude $|\hat{\beta}_{\mC j}^{M}|$ of conditional marginal regression coefficients over two groups of variables: 2:1999 and 2000. Lower left panel: the distributions of the magnitudes  $|\hat{\beta}^{M}_{\mC j}|$ of conditional marginal regression coefficients when five inactive variables are included in the conditioned set. Lower right panel: the distributions of the averages of the magnitude $|\hat{\beta}^{M}_{\mC j}|$ of conditional marginal regression coefficients given ten randomly chosen variables.}
\label{Fig2}
\end{figure}
Secondly, conditional screening helps for reducing the number of false negatives.  Marginal screening can fail when there are covariates in the non-active set that are highly correlated with active variables. To appreciate this, consider the linear model (\ref{eq1}) again with sparse regression coefficients $\bbeta^\star = (10, 0, \cdots, 0, 1)^T$, equi-correlation 0.9 among all covariates except $X_{2000}$, which is independent of the rest of the covariates. This setting gives
\[
 \Cov(X_1,Y)=10, \quad \Cov(X_{2000},Y)=1, \quad \mbox{and} \quad \Cov(X_j,Y)=9 \quad \mbox{for $j\neq 1,2000$}.
\]
In this case, marginal utilities for all nonactive variables are higher than that for the active variable $X_{2000}$.  A summary similar to Figure~\ref{Fig1} is shown in the upper left panel of Figure~\ref{Fig2}.  Therefore, based on SIS (sure independence screening) in Fan and Lv (2008), the active variable $X_{2000}$ has the least priority to be included.   By using the conditional screening approach in which the covariate $X_1$ is conditioned upon (used in the joint fit), marginal utilities of the spurious variables are significantly reduced.  The distributions of the average of the magnitude of the conditional fitted coefficients $\{|\hat{\beta}_{\mC j}^{M}|\}_{j=2}^{1999}$  and $|\beta_{\mC 2000}^{M}|$ are shown in the middle panel of Figure~\ref{Fig2}.  Clearly, the nonactive variables are significantly demoted by conditioning. %again more here
To observe effects of conditioning on extra variables and randomly chosen variables, a similar experiment to the first case is also done. Figure~\ref{Fig2c} depicts the distribution of the conditioned marginal fits when five extra variables are conditioned on. The contributions of variables $X_j$ in the presence of ten randomly chosen variables are given in Figure~\ref{Fig2d}. It is seen that, the relative magnitude of the hidden active variable $X_{2000}$ is considerably larger and hence it is more likely that it is recruited during screening.

Finally, as shown by \cite{SIS} and \cite{SISGLM}, for a given threshold of marginal utility,  the size of the selected variables depends on the correlation among covariates, as measured by the largest eigenvalue of $\bSigma$: $\lambda_{\max} \left(\bSigma\right)$.  The larger the quantity, the more variables have to be selected in order to have a sure screening property.  By using conditional screening, the relevant quantity now becomes $\lambda_{\max} \left(\bSigma_{\bsX_{\mD}|\bsX_{\mC}}\right)$, where $\bX_{\mC}$ refers to the $q$ covariates that we will condition upon and $\bX_{\mD}$ is the rest of the variables.  Conditioning helps reducing correlation among covariates $\bX_{\mD}$.  This is particularly the case when covariates $\bX$ share some common factors, as in many biological (e.g. treatment effects) and financial studies (e.g. market risk factors). To illustrate the benefits we consider the case where $\bX$ is given by equally correlated normal random variables. Simple calculations yield that $\lambda_{\max} \left(\bSigma_{\bsX_{\mD}}\right)=(1-r)+rd$ where $r$ is the common correlation
and $d = p-q$. As $\bX$ has a normal distribution, the conditional covariance matrix can be calculated easily and it can be shown that
\begin{equation} \label{eq2}
  \lambda_{\max} \left(\bSigma_{\bsX_{\mD}|\bsX_{\mC}}\right)=(1-r)+rd\frac{1-r}{1-r+rq}.
\end{equation}
Note that when $q = 0$, the formula reduces to the unconditional one.  It is clear that conditioning helps reducing the correlation among the variables.
To quantify the degree of de-correlation, Figure~\ref{Fig3} depicts the ratio
$\lambda_{\max} \left(\bSigma_{\bsX_{\mD}}\right) / \lambda_{\max} \left(\bSigma_{\bsX_{\mD}|\bsX_{\mC}}\right)$ as a function of $r$ for various choices of $q$ when $d=1000$.  The reduction is dramatic, in particular when $r$ is large or $q$ is large. The benefits of conditioning are clearly evidenced.

\begin{figure}[ht]
\centering
\includegraphics[width=0.8\textwidth]{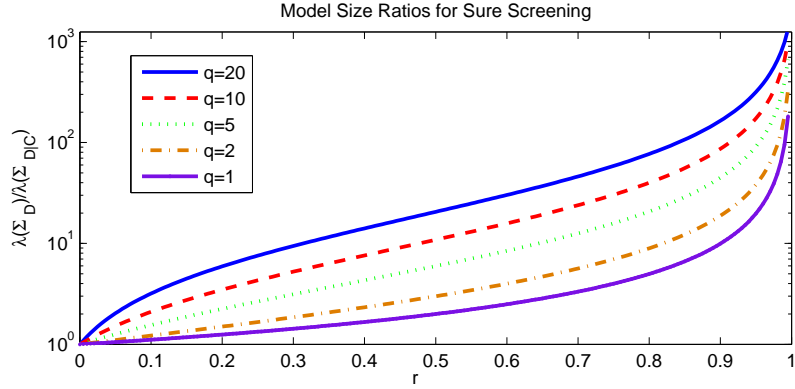}
\caption{Ratio of maximum eigenvalues of  unconditioned and conditioned covariance matrix.}
\label{Fig3}
\end{figure}

In this paper, we propose the conditional screening technique and formally establish the conditions under which it has a sure screening property.  We also give an upper bound for the number of selected variables for each given threshold value. Two data-driven methods for choosing the thresholding parameter are proposed to facilitate the practical use of the conditional screening technique.

The rest of the paper is organized as follows. In Section 2, we introduce the conditional sure independence screening procedure. The sure independence screening property and the uniform convergence of the conditional marginal maximum likelihood estimator are presented in Section 3. In Section 4, two approaches are proposed to choose the thresholding parameter for CSIS. Finally, we examine the performance of our procedure in Section 5 on simulated and real data. The details of the proofs are deferred to the Appendix.

\section{CONDITIONAL INDEPENDENCE SCREENING}

\subsection{Generalized Linear Models}

Generalized linear models assume that the conditional probability density of the random variable $Y$ given $\bX=\bx=(x_1\ldots,x_p)^T$ belongs to an exponential family
\begin{equation} \label{eq3}
f(y|\bx;\theta)=\exp\Big(y\theta(\bx) - b(\theta(\bx))+c(\bx;y)\Big),
\end{equation}
where $b(\cdot)$ and $c(\cdot)$ are specific known functions in the canonical parameter $\theta(\bx)$. Note that we ignore the dispersion parameter $\phi$,
since the interest only focuses on estimation of the mean regression function. However, it is easy to include a dispersion parameter $\phi$.  Under model (\ref{eq3}), we have the regression function
\[
\E(Y|\bX=\bx)=b'(\theta(\bx)).
\]
The canonical parameter is further parameterized as
\[
     \theta(\bx) = \bx^T \bbeta^\star,
\]
namely the canonical link is used in modeling the mean regression function.
Well known distributions in this exponential family include the normal, binomial, Poisson, and Gamma distributions.

In the ultrahigh dimensional sparse linear model, we assume that the true parameter $\bbeta^{\star}=(\beta_{1}^{\star},\ldots,\beta_{p}^{\star})^T$ is sparse.   Namely, the set
\[
   \mathcal{M}_{\star}=\{j=1,\ldots,p:\beta_{j}^{\star}\neq0\},
\]
is small. Our aim is to estimate the set $\mathcal{M}_{\star}$ and coefficient vector $\bbeta^\star$, as well as predicting the outcome $Y$.  This is a more challenging task than just predicting $Y$ as in many machine learning problems.
When the dimensionality is ultrahigh, one often employs a screening technique first to reduce the model size. It is particularly effective in distributed computation for dealing with ``Big Data''.

\subsection{Conditional Screening}

Conditional screening assumes that there is a set of variables $\bX_{\mC}$ that are known to be related to the response $Y$ and we wish to recruit additional variables from the rest of variables, given by $\bX_{\mD}$, to better explain the response variable $Y$.  For simplicity of notation, we assume without loss of generality that $\mC$ is the set of first $q$ variables and $\mD$ is the remaining set of $d = p-q$ variables.  We will use the notation
\[
   \bbeta_{\mC}=(\beta_{1},\ldots,\beta_{q})^{T}\in\R^{q}, \quad
   \mbox{and} \quad \bbeta_{\mD}=(\beta_{q+1},\ldots,\beta_{p})^{T}\in\R^{d},
\]
and similar notation for $\bX_{\mC}$ and $\bX_{\mD}$.

Assume without loss of generality that the covariates have been standardized so that
\[
\E(X_{j})=0\quad\text{and}\quad \E(X_{j}^{2})=1\quad\text{for}\ j \in \mD.
\]

Given a random sample $\{(\bX_{i},Y_{i})\}_{i=1}^n$ from the generalized linear model (\ref{eq3}) with the canonical link, the conditional maximum marginal likelihood estimator $\hat{\bbeta}_{\mC j}^{M}$ for $j=q+1,\ldots,p$ is defined as the minimizer of the (negative) marginal log-likelihood
\begin{equation}\label{eq4}
\hat{\bbeta}_{\mC j}^{M} =
\argmin_{\bsbeta_{\mC},\beta_{j}}\mathbb{P}_{n} \bigl \{l(\bX_{\mC}^{T}\bbeta_{\mC}+
X_{j}\beta_{j},Y) \bigr \},
\end{equation}
where $l(\theta,Y)= b(\theta) - \theta Y$ and $\mathbb{P}_{n} f(X,Y)=n^{-1}\sum_{i=1}^{n}f(X_{i},Y_{i})$ is the empirical measure. Denote from now on by $\hat{\beta}_{j}^M$  the last element of $\hat{\bbeta}_{\mC j}^M$.  It measures the strength of the conditional contribution of $X_j$ given $\bX_{\mC}$.  In the above notation, we assume that the intercept is used and is incorporated in the vector $\bX_{\mC}$.
Conditional marginal screening based on the estimated marginal magnitude is to keep the variables
\begin{equation}  \label{eq5}
     \hat{\mathcal{M}}_{\mathcal{D},\gamma}=\{ j \in \mD:  |\hat{\beta}_{j}^{M}|>\gamma\},
\end{equation}
for a given thresholding parameter $\gamma$.  Namely, we recruit variables with large additional contribution given $\bX_{\mC}$.  This method will be referred to as conditional sure independence screening (CSIS).  It depends, however, on the scale of $\E_L (X_j|\bX_{\mC})$ and
$\E_L (Y|\bX_{\mC})$ to be defined in Section 3.1.  A scale-free method is to use the likelihood reduction of the variable $X_j$ given $\bX_{\mC}$, which is equivalent to computing
\begin{equation} \label{eq6}
\hat{R}_{\mC j} = \min_{\bsbeta_{\mC},\beta_{j}}\mathbb{P}_{n} \bigl \{l(\bX_{\mC}^{T}\bbeta_{\mC}+ X_{j}\beta_{j},Y) \bigr \},
\end{equation}
after ignoring the common constant $\min_{\bsbeta_{\mC}}\mathbb{P}_{n} \bigl \{l(\bX_{\mC}^{T}\bbeta_{\mC},Y) \bigr \}$.  The smaller $\hat{R}_{\mC j}$, the more  the variable $X_j$ contributes in presence of $\bX_{\mC}$.  This leads to an alternative method based on the likelihood ratio statistics: recruit additional variables according to
\begin{equation}  \label{eq7}
     \tilde{\mathcal{M}}_{\mathcal{D},\tilde{\gamma}}=\{ j \in \mD:  \hat{R}_{\mC j} < \tilde \gamma\},
\end{equation}
where $\tilde \gamma$ is a thresholding parameter.  This method will be referred to as conditional maximum likelihood ratio screening (CMLR).

%again new stuff
We emphasize that, the set of variables $\bX_{\mC}$ does not necessarily have to contain active variables. Conditional screening only makes use of the fact that the effects of important variables are more visible in the presence of $\bX_{\mC}$ and the correlations of variables are weakened upon conditioning. This is commonly the case in many applications such as finance and biostatistics, where the variables share some common factors.  It gives hidden signature variables a chance to survive. In fact, it was demonstrated in the introduction that conditioning can be beneficial even if the set $\bX_{\mC}$ is chosen randomly.  Our theoretical study gives a formal justifications of the iterated method proposed in Fan and Lv (2008) and Fan {\em et. al.} (2009).
%end of new stuff

\section{SURE SCREENING PROPERTIES}

In order to prove the sure screening property of our method, we first need some properties on the population level.  Let $\bbeta_{\mC j}=(\bbeta_{\mC}^{T},\beta_{j})^{T}$, $\bX_{\mC j}=(\bX_{\mC}^{T},X_{j})^{T}$, and
\begin{equation}\label{eq8}
\bbeta_{\mC j}^{M} =\argmin_{\bsbeta_{\mC},\beta_{j}}
\E l(\bX_{\mC}^{T}\bbeta_{\mC}+X_{j}\beta_{j},Y),
\end{equation}
with the expectation taken under the true model.  Then, $\bbeta_{\mC j}^{M}$ is the population version of $\hat \bbeta_{\mC j}^{M}$.  To establish the sure screening property, we need to show that the marginal regression coefficient $\beta_{j}^{M}$, the last component of $\bbeta_{\mC j}^{M}$, provides useful probes for the variables in the joint model $\mathcal{M}_\star$ and
its sample version $\hat \beta_{ j}^{M}$ is uniformly close to the population counterpart $\beta_{ j}^{M}$.  Therefore, the vector of marginal fitted regression coefficients $\hat \bbeta_{\mC j}^{M}$ is useful for finding the variables in $\mathcal{M}_\star$.

\subsection{Properties on Population Level}

Since we are fitting $d$ marginal regressions, that is we are using only $q+1$ out of the $p$ original predictors, we need to introduce model misspecifications.  Thus, we do not expect that the marginal regression coefficient $\beta_{j}^{M}$ is equal to the joint regression parameter $\beta_j^\star$.  However, we hope that when the joint regression coefficient $|\beta_j^\star|$ exceeds a certain threshold,  $|\beta_{j}^{M}|$ exceeds another threshold in most cases.  Therefore, the marginal conditional regression coefficients provide useful probes for the joint regression.

By (\ref{eq8}), the marginal regression coefficients $\bbeta_{\mC j}^M$ satisfy the score equation
\begin{equation}  \label{eq9}
        \E b'(\bX_{\mC j}^T \bbeta_{\mC j}^M ) \bX_{\mC j} = \E Y \bX_{\mC j} = \E b'(\bX^T \bbeta^{\star}) \bX_{\mC j},
\end{equation}
where the second equality follows from the fact that $\E(Y|\bX ) = b'(\bX^T \bbeta^{\star})$.  Without using the additional variable $X_j$, the baseline parameter is given by
\begin{equation}\label{eq10}
\bbeta_{\mC}^{M} =\argmin_{\bsbeta_{\mC}} \E l(\bX_{\mC}^{T}\bbeta_{\mC},Y),
\end{equation}
and satisfies the equation
\begin{equation}  \label{eq11}
        \E b'(\bX_{\mC}^T \bbeta_{\mC }^M ) \bX_{\mC} = \E Y \bX_{\mC} = \E b'(\bX^T \bbeta^{\star}) \bX_{\mC}.
\end{equation}
We assume that the problems at marginal level are fully identifiable, namely, the solutions $\bbeta_{\mC}^M$  and  $\bbeta_{\mC j}^M$ are unique.

To understand the conditional contribution, we introduce the concept of the conditional linear expectation.  We use the notation
\begin{equation} \label{eq12}
\E_L (Y|\bX_{\mC}) = b'(\bX_{\mC}^{T}\bbeta_{\mC}^M), \quad \mbox{and} \quad \E_L (Y|\bX_{\mC j}) = b'(\bX_{\mC j}^{T}\bbeta_{\mC j}^M),
\end{equation}
which is the best linearly fitted regression within the class of linear functions.  Similarly, we use the notation $\E_L (X_j|\bX_{\mC})$ to denote the best linear regression fit of $X_j$ by using $\bX_{\mC}$.  Then, equation (\ref{eq11}) can be more intuitively expressed as
\begin{equation} \label{eq13}
    \E (Y - \E_L (Y | \bX_{\mC}) ) \bX_{\mC} = 0.
\end{equation}
Note that the conditioning in this paper is really a conditioning linear fit and the conditional expectation is really the conditional linear expectation.  This facilitates the implementation of the conditional (linear) screening in high-dimensional, but adds some technical challenges in the proof.

Let us examine the implication marginal signal, i.e. $\beta_{j}^M$.  When $\beta_j^M=0$, by (\ref{eq9}), the first $q$ components of $\bbeta_{\mC j}^M$, denoted by $\bbeta_{\mC j 1}^M$, should be equal to $\bbeta_{\mC}^M$ by uniqueness of equation (\ref{eq11}). Then, equation (\ref{eq9}) on the component $X_j$ entails
\[
   \E b'(\bX_{\mC}^T \bbeta_{\mC}^M ) X_{j} = \E Y X_{j}, \quad \mbox{or} \quad
   \E X_j (Y - \E_L (Y | \bX_{\mC}) ) = 0.
\]
Using (\ref{eq13}), the above condition can be more comprehensively expressed as
\begin{equation}  \label{eq14}
  \Cov_L \left(Y,X_{j}\big|\mathbf{X_{\mathcal{C}}}\right) \equiv \E (X_j - \E_L (X_j | \bX_{\mC}))  (Y - \E_L (Y | \bX_{\mC}) ) = 0.
\end{equation}
This proves the necessary condition of the following theorem.

\begin{thm}\label{thm1}
For $j \in \mD$, the marginal regression parameters $\beta_{j}^{M}=0$
if and only if $\Cov_L\left(Y,X_{j}\big|\mathbf{X_{\mathcal{C}}}\right)=0$.
\end{thm}

Proof of the sufficient part is given in Appendix \ref{app thm1}.
In order to have the sure screening property at the population level of equation (\ref{eq8}), the important variables $\{X_j, j \in \mathcal{M}_{\star\mathcal{D}}\}$ should be conditionally correlated with the response, where $\mathcal{M}_{\star\mathcal{D}} = \mathcal{M}_{\star} \cap \mathcal{D}$.
Moreover, if $X_j$ (with $j \in \mathcal{M}_{\star\mathcal{D}}$) is conditionally correlated with the response, the regression coefficient $\beta_j^M$ is non-vanishing. The sure screening property of conditional MLE (CMLE), given by equation \eqref{eq5}, will be guaranteed if the minimum marginal signal strength  is stronger than the estimation error.  This will be shown in Theorem \ref{thm2} and requires Condition \ref{cond1}. The details of the proof are relegated to Appendix \ref{app thm2}.

\begin{condition}\label{cond1}
\qquad \vspace{0mm}
\begin{enumerate}
\item [(i)] For $j\in\mathcal{M}_{\star\mathcal{D}}$, there exists a positive constant $c_{1}>0$ and $\kappa<1/2$ such that $\left|\Cov_L(Y,X_{j}|\bX_{\mC})\right|\geq c_{1} n^{-\kappa}$.

\item [(ii)] Let $m_j$ be the random variable defined by
\[
m_j=\frac{b' \lp \bX_{\mC j}^T\bbeta_{\mC j}^M\rp-b'\lp\bX_{\mC}^T\bbeta_{\mC}^M\rp}{\bX_{\mC j}^T\bbeta_{\mC j}^M- \bX_{\mC}^T\bbeta_{\mC}^M}.
\]
Then, $ \E m_j X_j^2 \leq c_2$ uniformly in $j=q+1,\ldots,p$.
\end{enumerate}
\end{condition}

Note that, by strict convexity of $b(\theta)$, $m_j>0$ almost surely.
When we are dealing with linear models, i.e. $b(\theta)=\theta^2/2$, then $m_j=1$ and Condition ~\ref{cond1}(ii) requires that $\mathbb{E} X_j^2$ is  bounded uniformly, which is automatically satisfied
by the normalization condition $\E X_j^2 = 1$.

\begin{thm}\label{thm2}
If Condition \ref{cond1} holds, then there exists a $c_3>0$ such that
\[
\min_{j\in\mathcal{M}_{\mathcal{D}\star}}\left|\beta_{j}^{M}\right|\geq c_{3}n^{-\kappa}.
\]
\end{thm}

\subsection{Properties on Sample Level}

In this section, we prove the uniform convergence of the conditional marginal maximum likelihood estimator and the sure screening property of the conditional sure independence screening method. In addition we provide an upper bound on the size of the set of selected variables $\hat{\mathcal{M}}_{\mathcal{D},\gamma}$.

Since the log-likelihood of a generalized linear model with the canonical link is concave, $\E(l(Y,\bX_{\mC j}^{T}\bbeta_{\mC j}))$ has a unique minimizer over $\bbeta_{\mC j}\in\mathcal{B}$ at an interior point $\bbeta_{\mC j}^{M}$, where $\mathcal{B}=\{|\beta_{1}^{M}|\leq B,\ldots,|\beta_{q}^{M}|\leq B,|\beta_{j}^{M}|\leq B\}$ is the set over which the marginal likelihood is maximized. To obtain the uniform convergence result at the sample level, a few more conditions on the conditional marginal likelihood are needed.

\begin{condition}\label{cond2}
\quad \vspace{0 mm}
\begin{enumerate}
\item [(i)] For the Fisher information $I_{j}(\bbeta_{\mC j})=\E(b''(\bX_{\mC j}^{T}\bbeta_{\mC j})\bX_{\mC j}\bX_{\mC j}^{T})$, its operator norm, $\|I_{j}(\bbeta_{\mC j})\|_{\mathcal{B}}$
is bounded, where
\[
\|I_{j}(\bbeta_{\mC j})\|_{\mathcal{B}}=\sup_{ {\footnotesize \bbeta_{\mC j}} \in\mathcal{B},\|\bx_{\mC j}\|=1}\|I_{j}(\bbeta_{\mC j})^{1/2}\bx_{\mC j}\|,
\]
and $\|\cdot\|$ is the Euclidian norm.

\item [(ii)] There exists some positive constants $r_{0},r_{1},s_{0},s_{1}$ and $\alpha$ such that for sufficiently large $t$
\[
P(|X_{j}|>t)\leq r_{1} \exp(-r_{0}t^{\alpha})\quad\text{for}\ j=1,\ldots,p
\]
and that
\[
\E\big(b(\bX^{T}\bbeta^{\star}+s_{0})-b(\bX^{T}\bbeta^{\star}))+\E\big(b(\bX^{T}\bbeta^{\star}-s_{0})-b(\bX^{T}\bbeta^{\star})\big)\leq s_{1}.
\]

\item [(iii)] The second derivative of $b(\theta)$
is continuous and positive. There exists an $\varepsilon_1>0$ such
that for all $j=q+1,\ldots,p$:
\[
\sup_{{\footnotesize \bbeta_{\mC j}}\in\mathcal{B},\|{\footnotesize \bbeta_{\mC j}} -{\footnotesize \bbeta_{\mC j}^M}\|\leq \varepsilon_1}|\E b(\bX_{\mC j}^{T}\bbeta_{\mC j})I(|X_{j}|>K_{n})|\leq o(n^{-1}),
\]
where $I(\cdot)$ is the indicator function and $K_n$ is an arbitrarily large constant such that for a given $\bbeta$ in $\mathcal{B}$, the function $l (\bx^T \bbeta, y)$ is Lipschitz for all $(\bx, y)$ in $\Lambda_n=\left\{ \mathbf{x},y:\left\Vert \mathbf{x}\right\Vert _{\infty}\leq K_{n}, |y|\leq K_{n}^{\star}\right\}$ with
$K_n^* = r_0 K_n^\alpha/s_0$.

\item [(iv)] For all $\bbeta_{\mC j}\in\mathcal{B}$, we have
\[
\E\big(l(\bX_{\mC j}^{T}\bbeta_{\mC j},Y)-l(\bX_{\mC j}^{T}\bbeta_{\mC j}^{M},Y)\big)\geq V\|\bbeta_{\mC j}-\bbeta_{\mC j}^{M}\|^{2},
\]
 for some positive $V$, bounded from below uniformly over $j=q+1,\ldots,p$.

\end{enumerate}
\end{condition}

The first three conditions given in Condition \ref{cond2} are satisfied for almost all of the commonly used generalized linear models. Examples include linear regression, logistic regression, and Poisson regression. The first part of Condition \ref{cond2}(ii) puts an exponential bound on the tails of $X_j$.

In the following theorem, the uniform convergence of our conditional marginal maximum likelihood estimator is stated as well as the sure screening property of the procedure. The proof of this theorem is deferred to Appendix \ref{app thm3}.

\begin{thm}\label{thm3}
Suppose that Condition \ref{cond2} holds. Let $k_n = b'(K_n B(q+1)) + r_0 K_n^{\alpha}/s_0$, with $K_n$ given in Condition \ref{cond2}.
\begin{enumerate}
\item [(i)] If $n^{1-2\kappa}k_{n}^{-2}K_{n}^{-2}\rightarrow\infty$, then for
any $c_{3}>0$, there exists a positive constant $c_{4}$ such that
\begin{eqnarray*}
    & & \mathbb{P}\left(\max_{q+1\leq j\leq p}|\hat{\beta}_{j}^{M}-\beta_{j}^{M}|\geq c_{3}n^{-\kappa}\right) \\
    & \leq & d \exp\big(-c_{4}n^{1-2\kappa}(k_{n}K_{n})^{-2}\big) +d nr_{2}\exp\big(-r_{0}K_{n}^{\alpha}\big),
\end{eqnarray*}
where $r_2 = q r_1 + s_1$.

\item [(ii)] If in addition, Condition \ref{cond1} holds, then by taking
$\gamma=c_{5}n^{-\kappa}$ with $c_{5}\leq c_{3}/2$, we have
\[
\mathbb{P}\left(\mathcal{M}_{\star\mathcal{D}}\subset\hat{\mathcal{M}}_{\mathcal{D},\gamma}\right)
 \geq 1-s \exp\big(-c_{4}n^{1-2\kappa}(k_{n}K_{n})^{-2}\big)- nr_{2} s \exp\big(-r_{0}K_{n}^{\alpha}\big),
\]
for some constant $c_5$, where $s=\left|\mathcal{M}_{\star\mathcal{D}}\right|$ the size of the set of nonsparse elements.
\end{enumerate}
\end{thm}

Note that the sure screening property, stated in the second conclusion of Theorem 3,
depends only on the size $s$ of the set of nonsparse elements and not on the dimensionality $d$ or $p$. This can be seen in the second conclusion above.  This result is understandable since we only need the elements in $\mathcal{M}_{\star\mathcal{D}}$ to pass the threshold, and this only requires the uniform convergence of $\hat{\beta}_j^M$ over $j \in \mathcal{M}_{\star\mathcal{D}}$.

The truncation parameter $K_n$ appears on both terms of the upper bound of the probability. There is a trade-off on this choice.  For the Bernoulli model with logistic link, $b'(\cdot)$ is bounded and the optimal order for $K_n$ is $n^{(1-2 \kappa)/(\alpha+2)}$.
In this case, the conditional sure independence screening method can handle the dimensionality
\[
    \log d = o\left(n^{(1-2\kappa)\alpha/(\alpha+2)}\right),
\]
which guarantees that the upper bound in Theorem~\ref{thm3} converges to zero.
A similar result for unconditional screening is shown in \cite{SISGLM}. In particular, when the covariates are bounded, we can take $\alpha = \infty$, and when covariates are normal, we have that $\alpha = 2$.  For the normal linear model, following the same argument as in \cite{SISGLM},  the optimal choice is $K_n=n^{\lp 1-2 \kappa \rp/A}$ where $A=\max \{\alpha+4, 3 \alpha +2\}$.
Then, conditional sure independence screening can handle dimensionality
\[
    \log d = o \lp n^{-(1-2 \kappa)\alpha/A} \rp,
\]
which is of order $o(n^{-\lp 1- 2 \kappa \rp /4})$ when $\alpha = 2$.

We have just stated the sure screening property of our CSIS method, that is $\hat{\mathcal{M}}_{\mathcal{D},\gamma}\supset \hat{\mathcal{M}}_{\star\mD}$. However, a good screening method does not only possess sure screening, but also retains a small set of variables after thresholding.  Below, we give a bound on the size of the selected set of variables, under the following additional conditions.

\begin{condition}\label{cond3}
\quad \vspace{0 mm}
\begin{enumerate}
\item [(i)] The variance $\Var(\bX^{T}\bbeta^{\star})={\bbeta^{\star}}^{T}\bSigma \bbeta^{\star}$ and $b''(\cdot)$ are bounded.

\item [(ii)] The minimum eigenvalue of the matrix $\E[m_j \bX_{\mC j}\bX_{\mC j}^T]$ is larger than a positive constant, uniformly over $j$, where $m_j$ is defined in Condition~\ref{cond1}(ii).

\item [(iii)] Letting
\[
\bZ=\E \Big\{ \E \big[\bX_{\mD} | \bX_{\mC}\big] \big[\bX^T \bbeta^\star -\bX_{\mC}^T \bbeta_{\mC}^M \big] \Big \},
\]
it holds that $\|\bZ\|_2^2 =o\Big \{ {\lambda_{\max} \big(\bSigma_{\mD| \mC} \big)}\Big\}$, with $\lambda_{\max}\big(\bSigma_{\mD|\mC} \big)$ the largest eigenvalue
of $\bSigma_{\mD|\mC} = \E [\bX_{\mD} -  \E_L(\bX_{\mD} | \bX_{\mC})] [\bX_{\mD} -  \E_L(\bX_{\mD} | \bX_{\mC})]^T$.
\end{enumerate}
\end{condition}

As noted above, for the normal linear model, $b(\theta)=\theta^2/2$. Condition~\ref{cond3} (ii) requires that the minimum eigenvalue of $\E \bX_{\mC j} \bX_{\mC j}^T$ be bounded away from zero.  In general,  by strict convexity of $b(\theta)$, $m_j>0$ almost surely.  Thus,  Condition~\ref{cond3}(ii) is mild.

For the linear model with $b'(\theta) = \theta$, by (\ref{eq11}),
\[
  \E  \bX_{\mC} \bX_{\mC}^T \bbeta_{\mC }^M = \E  \bX_{\mC} \bX^T \bbeta^{\star}
\]
and hence $\bZ = 0$ since $\E_L \big[\bX_{\mD} | \bX_{\mC}\big]$ is linear in $\bX_{\mC}$ by definition.  Thus, Condition~\ref{cond3}(ii) holds automatically.

From the proof of Theorem~\ref{thm4}, without Condition~\ref{cond3}(iii), Theorem~\ref{thm4} below continues to hold with
$\bSigma_{\mD|\mC}$ replaced by $\bSigma_{\mD|\mC} + \bZ \bZ^T$.

\begin{thm}\label{thm4} Under Conditions \ref{cond2} and \ref{cond3}, we have for $\gamma=c_{6}n^{-2\kappa}$, there exists a $c_{4}>0$ such that
\begin{eqnarray*}
& &\mathbb{P}\big(|\hat{\mathcal{M}}_{\mathcal{D},\gamma}|\leq O\big(n^{2\kappa}\lambda_{\max}\big(\bSigma_{\mD|\mC} \big)\big)\big)\\ &\geq & 1-d\Big(\exp\big(-c_{4}n^{1-2\kappa}(k_{n}K_{n})^{-2}\big)
 +nr_{2}\exp\big(-r_{0}K_{n}^{\alpha}\big)\Big).
\end{eqnarray*}

\end{thm}
This theorem is proved in Appendix \ref{app thm4}.

\section{SELECTION OF THE THRESHOLDING PARAMETER}\label{sec select gamma}

In the previous section, we have shown that CSIS has the sure screening property when the thresholding level $\gamma$
is chosen such that $\gamma\propto n^{-\kappa}$. Unfortunately, in practice $\gamma$, which relates to the minimum strength of marginal signals in the data, is always unknown. Therefore, $\gamma$ has to be estimated from the data itself. Underestimating $\gamma$  will result in a lot variables after screening, which leads to a large number of false positives, and similarly overestimation of $\gamma$ will prevent  sure screening.

In this section, we present two procedures that select a thresholding level for CSIS. The first approach is based on controlling the number of false positives by bounding the false discovery rate (FDR). This method uses the fact that quasi-likelihood estimates for GLMs enjoy asymptotic normality. The second approach, that we call random decoupling, uses a resampling technique to create the null model and to measure the maximum strength of noise. In random decoupling, we use marginal regression on the null model  to obtain the marginal regression coefficients that are known to be zero. We use the maximum of these marginal coefficients of the null model as a thresholding level.

\subsection{Controlling FDR}

It is well known that quasi-maximum likelihood estimates have an asymptotically
normal distribution under general conditions (\citealt{Heyde,Gao08}). Then, for covariates $j$ such that, $\beta_j^M=0$, asymptotically it follows that
\[
\left[I_{j}\left(\hat{\beta}_{j}^{M}\right)\right]^{1/2}\hat{\beta}_{j}^{M}\sim \mathcal{N}(0,1),
\]
where $I_{j}\left(\hat{\beta}_{j}^{M}\right)$ denotes the element that corresponds to $\beta_{j}$ in the information matrix $I_j(\bbeta_{\mC j})$.

Using this property, we can build a thresholding technique that bounds the proportion of elements $j$ such that, $\beta_j^M=0$. For the case, when $\beta_j^M=0$ for all $j\in (\mathcal{M}_{\star\mD})^c$, this rate is also called the false discovery rate in \cite{SISCox} and is given by $\mathbb{E}\left(\left|\hat{\mathcal{M}}_{\mD,\delta}\cap(\mathcal{M}_{\star\mD})^c\right|/\left|(\mathcal{M}_{\star\mD})^c\right|\right)$.

By choosing $\hat{\mathcal{M}}_{\mathcal{D},\delta}=\left\{ j:I_{j}\left(\hat{\beta}_{j}^{M}\right)^{1/2}\left|\hat{\beta}_{j}^{M}\right|\geq\delta\right\} $,
the expected false discovery rate is bounded above by $2\left(1-\Phi\left(\delta\right)\right)$, where $\Phi(\cdot)$ is the distribution function of a standard normal random variable.
This approach can also be seen as a modification of the method introduced
by \cite{SISCox} for the Cox model. By setting $\delta$ to
$\Phi^{-1}\left(1-f/(2d)\right)$ where $f$ is the maximum
number of false positives we can tolerate, we obtain an expected
false positive rate that is less than $f/(d-\left|\mathcal{M}_{\star \mathcal{D}}\right|)$
as the following theorem shows. The proof of this theorem is given in Appendix \ref{app thm FDR}.

\begin{condition}\label{cond4}
\quad \vspace{0 mm}
\begin{enumerate}

\item For any $j$, let $e_i=Y_i-b'(\bX_{i,\mC j}^T \bbeta_{\mC j})$ for $i=1,\dots,n$. For a given $j$, $\Var(e_i)\geq c_6$ for some positive $c_6$ and  $i=1,\ldots,n$ and $\sup_{i \geq 1} \E |e_i|^{2+\chi} < \infty$ for some $\chi>0$.

\item For $j\in\left(\mathcal{M_{\star\mathcal{D}}}\right)^c$, we have that $\Cov_L \left(Y,X_{j}\big|\mathbf{X_{\mathcal{C}}}\right)=0$.
\end{enumerate}
\end{condition}

\begin{thm}\label{thm FDR}
Under Conditions \ref{cond2}, \ref{cond3} and \ref{cond4}, if we choose
\[
\hat{\mathcal{M}}_{\mathcal{D},\delta}=\left\{ j:I_{j}\left(\hat{\beta}_{j}^{M}\right)^{1/2}\left|\hat{\beta}_{j}^{M}\right|\geq\delta\right\},
\]
where $\delta=\Phi^{-1}\left(1-f/(2d)\right)$ and $f$
is the number of false positives that can be tolerated, then, for some constant
$c_7>0$ it holds that
\[
\mathbb{E}\left(\frac{\left|\hat{\mathcal{M}}_{\mathcal{D},\delta}\cap (\mathcal{M}_{\star\mD})^c\right|}{\left|(\mathcal{M}_{\star\mD})^c\right|}\right)\leq\frac{f}{d}+\frac{c_7}{\sqrt{n}}.
\]
\end{thm}

\begin{comment}
To cover some of the cases where the assumptions do not hold, other small alterations can be made. One other possible approach that also considers FDR, as suggested by Wasserman and Roeder (2009), is to fit a model by selecting the first $n/log(n)$ variables that the screening suggests. Then, similar to our outlined FDR approach, a likelihood ratio test (e.g. $\chi^2$) is applied to select the final set. Depending on the nature of the problem, this might give better results as these tests are more stable once all of the active set is in the model. That is, by selecting $n/log(n)$ variables and fitting a model, one works with the full Fisher information matrix rather than the marginal one. Therefore, similar theoretical results can be obtained for this procedure with less strict assumptions.
\end{comment}

\subsection{Random Decoupling}

Random decoupling is an another procedure to select the thresholding parameter $\gamma$. It is used to create a null model, in which the data is formed by randomly permuting the rows of the last $d$ columns of the design matrix, while keeping the first $q$ columns of the design matrix intact. It is easy to see that by regressing $Y$ on $\bX_{\mC j}^*$ where the rows of the design matrix corresponding to $X_j$ ($j \not \in \mC$) have been randomly permuted, the obtained marginal values of $\hat{\beta}^{M*}_{j}$ is a statistical estimate of zero. These marginal estimates based on decoupled data measure the noise level of the estimates under the null model.  Let $\hat{\gamma}^* = \max_{q+1 \leq j \leq p} |\hat{\beta}^{M*}_{j}|$.  If $\hat{\gamma}^*$ is used as the thresholding value, all variables will be screened out based on the permuted data, which leads to no false positives in this case.  In other words, it is the minimum thresholding parameter that makes no false positives.  However, this $\hat{\gamma}^*$ depends on the realization of the permutation.  To stabilize the thresholding value, one can repeat this exercise $K$ times (e.g. 5 or 10 times), resulting in the values
\begin{equation} \label{eq15}
 \{|\hat{\beta}^{M*}_{kj}|, j=q+1, \cdots, p\}_{k=1}^K,
\end{equation}
$\{\gamma_k^*\}_{k=1}^K$, where $\gamma_k^* = \max_{q+1 \leq j \leq p} |\hat{\beta}^{M*}_{kj}|$.

Now, one can choose the maximum  of $\{\gamma_k^*\}_{k=1}^K$, denoted by $\hat \gamma_{\max}^*$,  as a thresholding value.  A more stable choice is the $\tau$-quantile of the values in (\ref{eq15}), denoted it by $\gamma^*_{\tau}$.  A useful range for $\tau$ is $[.95,1]$.  Note that for $\tau=1$,  $\gamma^*_{1} =\hat{\gamma}_{\max}^*$.  The selected variables are then
\[
\hat{\mathcal{M}}_{\mD,\tau} = \{j: |\hat{\beta}_j^{M}| \geq \gamma^*_{\tau} \}.
\]
In our numerical implementations, we do coupling five times, i.e. $K=5$, and take $\tau = 0.99$.  A similar idea for unconditional SIS appears already in \cite{AddSIS} for additive models.

\section{NUMERICAL STUDIES}

In this section, we demonstrate the performance of CSIS on simulated data and two empirical datasets. We compare CSIS versus sure independence screening and penalized least squares methods in a variety of settings.

\subsection{Simulation Study}

In the simulation study, we compare the performance of the proposed CSIS with Lasso (\citealt{Lasso}) and unconditional SIS (\citealt{SISGLM}), in terms of variable screening. We vary the sample size from $100$ to $500$ for different scenarios and the number of predictors range from $p=2,000$ to $40,000$. We present results with both the linear regression and the logistic regression.

We evaluate different screening methods on $200$ simulated data sets based on the following criteria:
\begin{enumerate}
\item MMMS: median minimum model size of the selected models that are required to have a sure screening.  The sampling variability of minimum model size (MMS) is measured by the robust standard deviation (RSD), which is defined as the associated interquartile range of MMS divided by $1.34$ across 200 simulations.
\item FP: average number of false positives across the 200 simulations,
\item FN: average number of false negatives across 200 simulations.
\end{enumerate}
We consider two different methods for selecting thresholding parameters: controlling FDR and random decoupling as outlined in the previous section, and we present false negatives and false positives for each method. Number of average false positives and false negatives are denoted by $\mathrm{FP}_{\pi}$ and $\mathrm{FN}_{\pi}$ for the random decoupling method and
$\mathrm{FP}_{\FDR}$ and $\mathrm{FN}_{\FDR}$ for the FDR method. For the FDR method, we have chosen the number of tolerated false positives as $n/\log{n}$.
For the experiments with $p=5,000$ and $p=40,000$, we do not report the corresponding results for Lasso, since it is not proposed for variable screening, and the data-driven choice of regularization parameter for model selection is not necessarily optimal for variable screening.

\subsubsection{Normal model}

The first two simulated examples concern linear models introduced in the introduction, regarding the false positives and false negatives of unconditional SIS.  We report the simulation results in Table~\ref{tab1} in which the column labeled ``{\bf Example 1}'' refers to the first setting and column labeled ``{\bf Example 2}'' referred to the second setting.
These examples are designed to fail the unconditional SIS.   Not surprisingly, SIS performs poorly in sure screening the variables, and conditional SIS easily resolves the problem. Also, we note that CSIS needs only one additional variable to have sure screening, whereas Lasso needs 15 additional variables.  Both the FDR and the random decoupling methods return no false negatives under almost all of the simulations.  In other words, both of the data-driven thresholding methods ensured the sure screening property.  However, they tend to be conservative, as the numbers of the false positives are high. The FDR approach has a relatively small number of false positives when used for conditional sure independent screening. For these settings, FDR method was found to be less conservative than the random decoupling method.

\begin{center}
\begin{table}[ht]
\begin{centering}
\caption{The MMMS, its RSD (in parentheses), the ``false
negative'' and ``false positive'' for the linear model with $n=100$ and $p=2,000$.}
\label{tab1}
\par\end{centering}

\begin{centering}
\begin{tabular}{cccccc}
\hline
 & \multicolumn{5}{c}{\textbf{Example 1}}\\
\hline
 & SIS & MLR & CSIS & CMLR & Lasso \\
\hline
{MMMS} & 1995 (0) & 1995 (0) & 1 (0) & 1 (0) & 16 (0) \\
\hline
$\mathrm{FP}_{\pi}$, $\mathrm{FN}_{\pi}$ & 1531,  0.07 & 1859, 1.00 & 175, 0 & 112, 0 & -\\
\hline
$\mathrm{FP}_{\FDR}$, $\mathrm{FN}_{\FDR}$ & 1934, 0.07 & - & 164, 0 & - & - \\
\hline\hline
 & \multicolumn{5}{c}{\textbf{Example 2}}\\
\hline
 & SIS & MLR & CSIS & CMLR & Lasso \\
\hline
{MMMS} & 1999 (0) & 1999 (0) & 1 (0) & 1 (0) & 16 (0) \\
\hline
$\mathrm{FP}_{\pi}$, $\mathrm{FN}_{\pi}$ & 1998,  0.01 & 1998, 0.04 & 543.1, 0 & 174, 0 & -\\
\hline
$\mathrm{FP}_{\FDR}$, $\mathrm{FN}_{\FDR}$ & 1998, 0.01 & - & 15.66, 0 & - & - \\
\hline
\end{tabular}
\par\end{centering}

\end{table}

\par\end{center}

In the next two settings, we work with higher dimensions, $p=5,000$ and $p=40,000$.
Following \cite{SISGLM}, we generate the covariates from
\begin{equation}\label{cov generated set1}
X_j = \frac{\varepsilon_j+a_j \varepsilon}{\sqrt{1+a_j^2}},
\end{equation}
where $\varepsilon$ and $\{\varepsilon_j\}_{j=1}^{p/3}$ are i.i.d. standard normal random variables, $\{\varepsilon_j\}_{j=p/3+1}^{2p/3}$ are i.i.d. double exponential variables with location parameter zero and scale parameter one and $\{\varepsilon_j\}_{j=2p/3 +1 }^{p}$ are i.i.d. and follow a mixture normal distribution with two components $N(-1,1)$, $N(1,0.5)$ and equal mixture proportion. The covariates are standardized to have mean zero and variance one.
Specifically, we consider the following two settings.

{\bf Example 3}.  In this setting, $p=5,000$ and $s=12$. The constants $a_1,\ldots,a_{100}$ are the same and chosen such that the correlation $\rho=\Corr(X_i,X_j)=0, 0.2, 0.4, 0.6$ and $0.8$ among the first 100 variables and $a_{101}=\ldots=a_{5,000}=0$.

{\bf Example 4}.  In this setting, $p=40,000$ and $s=6$. The constants $a_1,\ldots,a_{50}$ are generated from the normal random distribution with mean $a$ and variance $1$ and $a_{51}=\ldots,a_{40,000}=0$. The constant $a$ is taken such that $\mathbb{E}(\Corr(X_i,X_j))=0, 0.2, 0.4, 0.6$ and $0.8$ among the first $r$ variables.

In both of the settings $\bbeta^\star$ is generated from an alternating sequence of $1$ and $1.3$. For conditional sure independence screening, we condition on the first 2 covariates if $s=6$ and we condition on the first 4 covariates if $s=12$.
Results are presented in Tables \ref{tab2} and \ref{tab3}.\\

\begin{table}
\caption{The MMMS, its RSD (in parentheses), the ``false
positive'' and ``false negative'' for Example 3 with $p=5,000$ and $s = 4 + 8$.}
\label{tab2}
\begin{center}
\begin{tabular}{ccccccc}
\hline
 \multicolumn{7}{c}{\textbf{Sure Independence Screening}}\\
\hline
$\rho$ & $n$ & $\mathrm{MMMS}$ & $\mathrm{FP}_{\pi}$ & $\mathrm{FN}_{\pi}$ & $\mathrm{FP}_{\FDR}$ & $\mathrm{FN}_{\FDR}$\\  \hline
0.00  & 300 &  86 (150)  &  0.21  &  4.61  &  20.75  &  1.23 \\
0.20  & 100 &  43 (19)  &  34.17  &  0.82  &  87.70 &  0.03 \\
0.40  & 100 &  56 (20)  &  87.38  &  0.00  &  101.75 &  0.00 \\
0.60  & 100 &  58 (24)  &  88.20  &  0.00  &  101.68  &  0.00 \\
0.80  & 100 &  63 (19)  &  88.17  &  0.00  &  101.64  &  0.00 \\
\hline
\multicolumn{7}{c}{\textbf{Conditional Sure Independence Screening}}\\
\hline
$\rho$ & $n$ & $\mathrm{MMMS}$ & $\mathrm{FP}_{\pi}$ & $\mathrm{FN}_{\pi}$ & $\mathrm{FP}_{\FDR}$ & $\mathrm{FN}_{\FDR}$\\
\hline
0.00  & 300 &  57 (92)  &  0.16  &  3.74  &  21.09  &  0.97 \\
0.20  & 100 &  31 (38)  &  2.74  &  2.97  &  29.93  &  0.69 \\
0.40  & 100 &  29 (21)  &  17.65  &  0.99  &  48.03  &  0.42 \\
0.60  & 100 &  32 (18)  &  44.93  &  0.23  &  55.60  &  0.29 \\
0.80  & 100 &  42 (20)  &  67.55 &  0.06  &  50.01 &  0.66 \\
\hline
\end{tabular}
\end{center}

\begin{center}
\begin{tabular}{ccccc}
\hline
\multicolumn{5}{c}{\textbf{Maximum Likelihood Ratio}}\\
\hline
$\rho$ & $n$ & $\mathrm{MMMS}$ & $\mathrm{FP}_{\pi}$ & $\mathrm{FN}_{\pi}$\\
\hline
0.00  & 300 &  86 (141)  &  0.77 & 0.23\\
0.20  & 100 &  43 (20)  & 47.88 & 0.03\\
0.40  & 100 &  52 (19) & 88.48 & 0.00 \\
0.60  & 100 &  58 (18) & 88.78 & 0.00\\
0.80  & 100 &  60 (19)  & 88.75 & 0.00 \\
\hline
\multicolumn{5}{c}{\textbf{Conditional Maximum Likelihood Ratio}}\\
\hline
$\rho$ & $n$ & $\mathrm{MMMS}$ & $\mathrm{FP}_{\pi}$ & $\mathrm{FN}_{\pi}$\\
\hline
0.00  & 300 &  18 (25)  & 0.72 & 1.65 \\
0.20  & 100 &  23 (24)  & 5.71 & 1.44 \\
0.40  & 100 &  23 (17)  & 16.45 & 0.76 \\
0.60  & 100 &  28 (19)  & 23.81 & 0.55 \\
0.80  & 100 &  33 (22)  & 26.09 & 0.69 \\
\hline
\end{tabular}
\end{center}
\end{table}

\begin{table}
\caption{The MMMS, its RSD (in parentheses), the ``false
positive'' and ``false negative'' for Example 4 with $p=40,000$
and $s = 2 + 4$.}
\label{tab3}
\begin{center}
\begin{tabular}{ccccccc}
\hline
 \multicolumn{7}{c}{\textbf{Sure Independence Screening}}\\
\hline
$\rho$ & $n$ & $\mathrm{MMMS}$ & $\mathrm{FP}_{\pi}$ & $\mathrm{FN}_{\pi}$ & $\mathrm{FP}_{\FDR}$ & $\mathrm{FN}_{\FDR}$\\  \hline
0.00  & 200 &  1133 (8246)  &  11.46  &  1.35  &  40.70  &  0.89 \\
0.20  & 200 &  37 (1079)  &  30.37  &  0.61  &  57.83  &  0.46 \\
0.40  & 200 &  37 (12)  &  37.92  &  0.32  &  62.71  &  0.24 \\
0.60  & 200 &  37 (11)  &  41.35  &  0.17  &  65.61  &  0.13 \\
0.80  & 200 &  36 (12)  &  43.73  &  0.02  &  66.89  &  0.02 \\
\hline
  \multicolumn{7}{c}{\textbf{Conditional Sure Independence Screening}}\\
\hline
$\rho$ & $n$ & $\mathrm{MMMS}$ & $\mathrm{FP}_{\pi}$ & $\mathrm{FN}_{\pi}$ & $\mathrm{FP}_{\FDR}$ & $\mathrm{FN}_{\FDR}$\\  \hline
0.00  & 200 &  13 (84)  &  5.83  &  0.57  &  31.04  &  0.43 \\
0.20  & 200 &  16 (18)  &  16.62  &  0.31  &  41.07  &  0.23 \\
0.40  & 200 &  16 (12)  &  23.89  &  0.11  &  45.61  &  0.08 \\
0.60  & 200 &  17 (10)  &  29.83  &  0.03  &  50.05  &  0.01 \\
0.80  & 200 &  17 (10)  &  37.41  &  0.00  &  54.34  &  0.02 \\
\hline
\end{tabular}
\end{center}

\begin{center}
\begin{tabular}{ccccc}
\hline
\multicolumn{5}{c}{\textbf{Maximum Likelihood Ratio}}\\
\hline
$\rho$ & $n$ & $\mathrm{MMMS}$ & $\mathrm{FP}_{\pi}$ & $\mathrm{FN}_{\pi}$\\
\hline
0.00  & 200 &  1133 (8246) & 13.61 &  0.19 \\
0.20  & 200 &  41 (1503)  & 31.62 &  0.11  \\
0.40  & 200 &  37 (12)  & 39.24  &   0.06 \\
0.60  & 200 &  37 (11)  & 42.51 & 0.05  \\
0.80  & 200 &  36 (12)  & 44.45 & 0.00 \\
\hline
\multicolumn{5}{c}{\textbf{Conditional Maximum Likelihood Ratio}}\\
\hline
$\rho$ & $n$ & $\mathrm{MMMS}$ & $\mathrm{FP}_{\pi}$ & $\mathrm{FN}_{\pi}$\\
\hline
0.00  & 200 &  14 (261)  & 5.42 & 0.07  \\
0.20  & 200 &  10 (21)  & 13.02 & 0.05 \\
0.40  & 200 &  7 (10)  & 18.04 & 0.02 \\
0.60  & 200 &  6 (5)  & 21.66 & 0.01 \\
0.80  & 200 & 6 (3)  & 25.00 & 0.00 \\
\hline
\end{tabular}
\end{center}
\end{table}

As expected, CSIS needs a smaller model size to have all the relevant variables, i.e. to possess the sure screening property. The effect is more pronounced for higher $p$ and when more of the variables are correlated. A surprising result is that the advantage of conditioning is less when the correlation levels are higher. This is probably because of the fact that only 50 or 100 of the covariates are correlated, hence conditioning cannot fully utilize its advantages. We also see that, both methods for choosing the thresholding parameter are very effective. Both the FDR and empirical decoupling methods tend to have the sure screening property (no false negatives) and low number of false positives.

\subsubsection{Binomial model}

In this section data are given by i.i.d. copies of $(\bX^T,Y)$, where the conditional distribution of $Y$ given $\bX=\bx$ is a binomial distribution with probability of success $\mathbb{P}(\bx) = \exp\left( \bx^T \bbeta^\star\right) \left(1+ \exp\left(\bx^T \bbeta^\star \right) \right)^{-1}.$ The first two settings use the same setup of covariates  and the same values for $\bbeta^\ast$ as that in Example 1. The results are given in Table \ref{tab bin1}.

The results are almost the same as in the normal model. Conditional screening always lists the active variable as the most important one and Lasso only needs 16 variables. We also see that FDR and random decoupling methods are still successful, even though the setting is nonlinear.

The final settings for the binomial model use the same construction for the covariates as those in Examples 3 and 4. We again work with $s=6$ and $s=12$. For settings 2 and 3, $\bbeta^\star$ is again given by a sequence of $1$s and $1.3$s. Results are given in Tables \ref{tab bin2} and \ref{tab bin3}.

The results are the same as for the normal model. Due to the nonlinear nature of the problem, the minimum model size is slightly higher and the thresholding methods are less efficient. However, even though the covariates are not too correlated, overall advantage of conditional sure independence screening can easily be observed.

\begin{center}
\begin{table}[h]
\begin{centering}
\caption{The MMMS, its RSD (in parentheses) for the binomial model with the
{}``false negative'' and {}``false positive'' settings for $n=100$
and $p=2,000$.}
\label{tab bin1}
\par\end{centering}

\begin{centering}
\begin{tabular}{cccccc}
\hline
 & \multicolumn{5}{c}{\textbf{Example 1}}\\
\hline
 & SIS & MLR & CSIS & CMLR & Lasso \\
\hline
{MMMS} & 1995 (1.5) & 1995 (1.5) & 1 (0) & 1 (0) & 16 (0) \\
\hline
$\mathrm{FP}_{\pi}$, $\mathrm{FN}_{\pi}$ & 726,  0.07 & 1282, 1.00 & 35.72, 0 & 31.11, 0.01& -\\
\hline
$\mathrm{FP}_{\FDR}$, $\mathrm{FN}_{\FDR}$ & 1344, 0.07 & - & 34.05, 0 & - & - \\
\hline\hline
 & \multicolumn{5}{c}{\textbf{Example 2}}\\
\hline
 & SIS & MLR & CSIS & CMLR & Lasso \\
\hline
{MMMS} & 1999 (0) & 1999 (0) & 1 (0) & 1(0) & 16 (0) \\
\hline
$\mathrm{FP}_{\pi}$, $\mathrm{FN}_{\pi}$ & 1998,  0.03 &  1998, 0.14 & 462, 0 & 157, 0.01 & -\\
\hline
$\mathrm{FP}_{\FDR}$, $\mathrm{FN}_{\FDR}$ & 1998, 0.04 & - & 5.65, 0 & - & - \\
\hline
\end{tabular}
\par\end{centering}
\end{table}
\par\end{center}

\begin{table}
\caption{The MMMS, its RSD (in parentheses), the ``false
positive'' and ``false negative'' for Example 3 with the binomial model with $p=5,000$
and $s = 4 + 8$.}
\label{tab bin2}
\begin{center}
\begin{tabular}{ccccccc}
\hline
\multicolumn{7}{c}{\textbf{Sure Independence Screening}}\\
\hline
$\rho$ & $n$ & $\mathrm{MMMS}$ & $\mathrm{FP}_{\pi}$ & $\mathrm{FN}_{\pi}$ & $\mathrm{FP}_{\FDR}$ & $\mathrm{FN}_{\FDR}$\\
\hline
0.00  & 300 &  215 (312)  &  0.19  &  5.78  &  23.06  &  1.77 \\
0.20  & 300 &  27 (14)  &  73.22 &  0.02  &  109.56  &  0.00 \\
0.40  & 300 &  49 (21)  &  88.19  &  0.00  &  110.15  &  0.00 \\
0.60  & 300 &  56 (20)  &  88.17  &  0.00  &  110.00 &  0.00 \\
0.80  & 300 &  68 (19)  &  88.20  &  0.00  &  110.34  &  0.00 \\
\hline
\multicolumn{7}{c}{\textbf{Conditional Sure Independence Screening}}\\
\hline
$\rho$ & $n$ & $\mathrm{MMMS}$ & $\mathrm{FP}_{\pi}$ & $\mathrm{FN}_{\pi}$ & $\mathrm{FP}_{\FDR}$ & $\mathrm{FN}_{\FDR}$\\
\hline
0.00  & 300 &  87 (173)  &  20.15 & 1.24 &  24.03 &  1.11 \\
0.20  & 300 &  19 (13)  &  49.25 &  0.14 &  53.87 &  0.11 \\
0.40  & 300 &  34 (23)  &  67.82  &  0.17 &  61.72  &  0.31 \\
0.60  & 300 &  43 (24)  &  77.36  &  0.21  &  53.83  &  1.01 \\
0.80  & 300 &  66 (55)  &  78.33  &  0.51 &  36.16  &  3.42 \\
\hline
\end{tabular}
\end{center}
\begin{center}
\begin{tabular}{ccccc}
\hline
\multicolumn{5}{c}{\textbf{Maximum Likelihood Ratio}}\\
\hline
$\rho$ & $n$ & $\mathrm{MMMS}$ & $\mathrm{FP}_{\pi}$ & $\mathrm{FN}_{\pi}$\\
\hline
0.00  & 300 &  210 (312)  & 20.18 &  0.08\\
0.20  & 300 &  28 (17)  & 107.08 & 0.00 \\
0.40  & 300 &  47 (24)  & 107.82 &  0.00 \\
0.60  & 300 &  60 (22)  & 107.47 & 0.00 \\
0.80  & 300 &  67 (19)  & 107.30 & 0.00 \\
\hline
\multicolumn{5}{c}{\textbf{Conditional Maximum Likelihood Ratio}}\\
\hline
$\rho$ & $n$ & $\mathrm{MMMS}$ & $\mathrm{FP}_{\pi}$ & $\mathrm{FN}_{\pi}$\\
\hline
0.00  & 300 &  83 (173)  & 20.18 & 1.21 \\
0.20  & 300 &  20 (14)  & 45.27 & 0.20 \\
0.40  & 300 & 39 (30)  & 53.48 & 0.49 \\
0.60  & 300 &  71 (87)  & 49.47 & 1.15 \\
0.80  & 300 &  402 (561) & 35.42 & 3.43 \\
\hline
\end{tabular}
\end{center}
\end{table}

\begin{table}
\caption{The MMMS, its RSD (in parentheses), the ``false
positive'' and ``false negative'' for Example 4 with the binomial model with $p=40,000$
and $s = 2 + 4$.}
\label{tab bin3}
\begin{center}
\begin{tabular}{ccccccc}
\hline
\multicolumn{7}{c}{\textbf{Sure Independence Screening}}\\
\hline
$\rho$ & $n$ & $\mathrm{MMMS}$ & $\mathrm{FP}_{\pi}$ & $\mathrm{FN}_{\pi}$ & $\mathrm{FP}_{\FDR}$ & $\mathrm{FN}_{\FDR}$\\
\hline
0.00  & 500 &  318 (7038)  &  12.04  &  1.22  &  51.32  &  0.79 \\
0.20  & 500 &  38 (428)  &  32.47  &  0.57  &  68.46  &  0.38 \\
0.40  & 500 &  38 (12)  &  38.66  &  0.27  &  73.42  &  0.19 \\
0.60  & 500 &  38 (12)  &  41.99  &  0.16  &  76.11  &  0.10 \\
0.80  & 500 &  35 (12)  &  43.84  &  0.03  &  77.38  &  0.02 \\
\hline
\multicolumn{7}{c}{\textbf{Conditional Sure Independence Screening}}\\
\hline
$\rho$ & $n$ & $\mathrm{MMMS}$ & $\mathrm{FP}_{\pi}$ & $\mathrm{FN}_{\pi}$ & $\mathrm{FP}_{\FDR}$ & $\mathrm{FN}_{\FDR}$\\
\hline
0.00  & 500 &  13 (354)  &  5.96  &  0.66  &  42.51  &  0.49 \\
0.20  & 500 &  15 (16)  &  14.51  &  0.39  &  49.79  &  0.27 \\
0.40  & 500 &  16 (13)  &  19.11  &  0.24  &  51.68  &  0.22 \\
0.60  & 500 &  19 (10)  &  22.80  &  0.21  &  51.78  &  0.24 \\
0.80  & 500 &  19 (10)  &  26.39  &  0.14  &  46.49  &  0.64 \\
\hline
\end{tabular}
\end{center}
\begin{center}
\begin{tabular}{ccccc}
\hline
\multicolumn{5}{c}{\textbf{Maximum Likelihood Ratio}}\\
\hline
$\rho$ & $n$ & $\mathrm{MMMS}$ & $\mathrm{FP}_{\pi}$ & $\mathrm{FN}_{\pi}$\\
\hline
0.00  & 500 &  309 (7030)  & 14.06 & 0.22 \\
0.20  & 500 &  37 (255)  & 34.10 & 0.09  \\
0.40  & 500 &  35.5 (11)  & 40.50 & 0.05 \\
0.60  & 500 &  35.5 (12)  & 42.89 &  0.03 \\
0.80  & 500 &  33.5 (14)  & 44.39 & 0.00 \\
\hline
\multicolumn{5}{c}{\textbf{Conditional Maximum Likelihood Ratio}}\\
\hline
$\rho$ & $n$ & $\mathrm{MMMS}$ & $\mathrm{FP}_{\pi}$ & $\mathrm{FN}_{\pi}$\\
\hline
0.00  & 500 & 25 (892)  & 5.96 & 0.14 \\
0.20  & 500 & 13 (62)  & 12.38 & 0.09 \\
0.40  & 500 &  13 (22)  & 14.17 &  0.08 \\
0.60  & 500 & 15.5 (17)  & 13.75 & 0.11 \\
0.80  & 500 &  22 (72)  & 9.30 & 0.28 \\
\hline
\end{tabular}
\end{center}
\vspace{2 mm}
\end{table}

%obviously new section

\subsubsection{Robustness of CSIS}

In this section, we evaluate the performance of CSIS under three different conditioning sets: The set consists of (i) only active variables, (ii) both active  and inactive variables and (iii) only (randomly chosen) inactive variables. We consider a different correlation structure where the number of correlated variables is significantly large. %In this setting, the disadvantages of conditioning on an inactive variable will be more significant, since a wrongly conditioned variable is more likely to hide the contribution of an important variable.

For this experiment, Example 5, we set $p=10,1000$ and $s=6$. We generate covariates from equation \eqref{cov generated set1} and choose the constants $a_1,\dots,a_{2000}$ such that the correlation $\rho=\Corr(X_i,X_j)=0, 0.2, 0.4, 0.6$ and $0.8$ among the first 2000 variables and $a_{2001}=\ldots=a_{10,000}=0$. We fix $\beta^{\star}=\{1,2,1,2,0,\dots,0,1,2\}^T$.

The following three conditioning sets are considered (i) $\mC_{ 1}=\{1,2\}$; (ii) $\mC_{2}=\{1,2,5,2001\}$ and (iii) $\mC_{3} = $\{random choice of 4 inactive variables\}.  More precisely, $\mC_{3}$ consists of 3 randomly chosen variables from the first two thousand variables which are correlated and 1 randomly chosen inactive variable from the rest.  Note that variables 1 and 2 are active variables whereas variables 5 and 2001 are inactive. We have simulation results using both the conditional MLE (\ref{eq5}) and conditional MLR (\ref{eq6}).  To save the space, we only present the results using the conditional MLE for the normal model in Table \ref{tabI1} and for the binomial model in Table \ref{tabI3}.

The results show clearly that the benefits of conditional screening are significant even when variables are wrongly chosen. CSIS reduces the minimum model size at least by half, and for most of the cases it uses 10 times as less variables as the unconditioning one. CSIS performs well even if some of the conditioned variables are inactive or even all are randomly selected inactive variables.   For the worst cases, ``mis-conditioning" forced CSIS to recruit twice as many variables, and for most of the cases, the difference is not excessive.
In all cases, CSIS performs significantly better than the unconditioning case.
%For instance, for the Binomial model, conditional maximum likelihood ratio screening with extra inactive variables requires at most 70 extra variables to have sure screening. On the other hand, CSIS with random conditioning is strictly better than marginal screening. Furthermore, there is a tremendous performance difference between the maximum likelihood ratio test (MLR) and conditioned MLR (CMLR) with random variables. For the linear model setting, MLR needs to recruit 10 times as CMLR in the highly correlated setting, and for the binomial model, MLR needs twice as many features to have sure screening.

%An another observation is that, the permutation test for choosing the thresholding level works well even under this extremely noisy setting. Unfortunately, the method based on controlling the FDR performs very poorly mainly because the level of tolerated false positives is too low in order to have sure screening. To summarize, numerical results in this section demonstrate that conditional sure independence screening is useful even if the all of the variables in the conditioned set are not active or if the conditioned variables are chosen randomly.

\begin{table}
\caption{The MMMS, its RSD (in parentheses), the ``false
positive'' and ``false negative'' for Example 5 for the Linear Model with $p=10,000$ and $s = 2 + 4$.}
\label{tabI1}
\begin{center}
\begin{tabular}{ccccccc}
\hline
 \multicolumn{7}{c}{\textbf{Sure Independence Screening}}\\
\hline
$\rho$ & $n$ & $\mathrm{MMMS}$ & $\mathrm{FP}_{\pi}$ & $\mathrm{FN}_{\pi}$ & $\mathrm{FP}_{\FDR}$ & $\mathrm{FN}_{\FDR}$\\  \hline
0.00 & 200 & 35 (80) & 98.20 & 0.28 & 20.16 & 0.63 \\
0.20 & 200 & 1601 (812) & 1854.75 & 0.34 & 1537.35 & 0.51 \\
0.40 & 200 & 2038 (267) & 2083.30 & 0.45 & 2010.73 & 0.63 \\
0.60 & 200 & 2108 (470) & 2088.11 & 0.52 & 2010.59 & 0.73 \\
0.80 & 200 & 2193 (663) & 2092.08 & 0.58 & 2010.59 & 0.83 \\
\hline
\multicolumn{7}{c}{\textbf{CSIS with $\mC_1$}}\\
\hline
$\rho$ & $n$ & $\mathrm{MMMS}$ & $\mathrm{FP}_{\pi}$ & $\mathrm{FN}_{\pi}$ & $\mathrm{FP}_{\FDR}$ & $\mathrm{FN}_{\FDR}$\\
\hline
0.00 & 200 & 6 (8) & 98.17 & 0.07 & 23.51 & 4.00 \\
0.20 & 200 & 13 (47) & 440.33 & 0.04 & 143.85 & 3.90 \\
0.40 & 200 & 75 (215) & 1001.84 & 0.03 & 336.05 & 3.67 \\
0.60 & 200 & 216 (358) & 1372.48 & 0.01 & 379.81 & 3.64 \\
0.80 & 200 & 423 (429) & 1518.04 & 0.00 & 234.19 & 3.79 \\
\hline
\multicolumn{7}{c}{\textbf{CSIS with $\mC_2$}}\\
\hline
$\rho$ & $n$ & $\mathrm{MMMS}$ & $\mathrm{FP}_{\pi}$ & $\mathrm{FN}_{\pi}$ & $\mathrm{FP}_{\FDR}$ & $\mathrm{FN}_{\FDR}$\\
\hline
0.00 & 200 & 6 (7) & 98.29 & 0.08 & 23.44 & 4.00 \\
0.20 & 200 & 21 (75) & 565.76 & 0.03 & 212.80 & 3.75 \\
0.40 & 200 & 152 (413) & 1367.95 & 0.03 & 642.06 & 3.33 \\
0.60 & 200 & 443 (676) & 1766.88 & 0.01 & 830.50 & 3.12 \\
0.80 & 200 & 868 (643) & 1860.01 & 0.00 & 594.86 & 3.40 \\
\hline
\multicolumn{7}{c}{\textbf{CSIS with $\mC_3$}}\\
\hline
$\rho$ & $n$ & $\mathrm{MMMS}$ & $\mathrm{FP}_{\pi}$ & $\mathrm{FN}_{\pi}$ & $\mathrm{FP}_{\FDR}$ & $\mathrm{FN}_{\FDR}$\\
\hline
0.00 & 200 & 44 (90) & 100.33 & 0.30 & 23.23 & 2.31 \\
0.20 & 200 & 481 (687) & 1022.85 & 0.24 & 499.31 & 1.50 \\
0.40 & 200 & 1322 (752) & 1806.40 & 0.20 & 1147.03 & 0.86 \\
0.60 & 200 & 1652 (462) & 2003.43 & 0.10 & 1345.32 & 0.63 \\
0.80 & 200 & 1716 (297) & 2037.08 & 0.03 & 1103.83 & 0.94 \\
\hline

\end{tabular}
\end{center}
\end{table}

\begin{table}
\caption{The MMMS, its RSD (in parentheses), the ``false
positive'' and ``false negative'' for Example 5 for the Binomial Model with $p=10,000$ and $s = 2 + 4$.}
\label{tabI3}
\begin{center}
\begin{tabular}{ccccccc}
\hline
 \multicolumn{7}{c}{\textbf{Sure Independence Screening}}\\
\hline
$\rho$ & $n$ & $\mathrm{MMMS}$ & $\mathrm{FP}_{\pi}$ & $\mathrm{FN}_{\pi}$ & $\mathrm{FP}_{\FDR}$ & $\mathrm{FN}_{\FDR}$\\  \hline
0.00 & 400 & 24 (59) & 97.39 & 0.21 & 27.29 & 0.48 \\
0.20 & 400 & 1606 (776) & 1933.60 & 0.20 & 1725.60 & 0.39 \\
0.40 & 400 & 2029 (101) & 2082.82 & 0.30 & 2016.35 & 0.52 \\
0.60 & 400 & 2070 (258) & 2087.22 & 0.45 & 2015.59 & 0.64 \\
0.80 & 400 & 2096 (429) & 2090.86 & 0.51 & 2015.07 & 0.66 \\
\hline
\multicolumn{7}{c}{\textbf{CSIS with $\mC_1$}}\\
\hline
$\rho$ & $n$ & $\mathrm{MMMS}$ & $\mathrm{FP}_{\pi}$ & $\mathrm{FN}_{\pi}$ & $\mathrm{FP}_{\FDR}$ & $\mathrm{FN}_{\FDR}$\\
\hline
0.00 & 400 & 8 (16) & 98.20 & 0.10 & 31.98 & 4.00 \\
0.20 & 400 & 22 (75) & 361.04 & 0.10 & 138.73 & 3.85 \\
0.40 & 400 & 107 (223) & 743.80 & 0.08 & 247.20 & 3.74 \\
0.60 & 400 & 289 (439) & 1022.71 & 0.10 & 246.67 & 3.75 \\
0.80 & 400 & 637 (528) & 1142.79 & 0.16 & 133.97 & 3.82 \\
\hline
\multicolumn{7}{c}{\textbf{CSIS with $\mC_2$}}\\
\hline
$\rho$ & $n$ & $\mathrm{MMMS}$ & $\mathrm{FP}_{\pi}$ & $\mathrm{FN}_{\pi}$ & $\mathrm{FP}_{\FDR}$ & $\mathrm{FN}_{\FDR}$\\
\hline
0.00 & 400 & 7 (17) & 98.33 & 0.11 & 31.31 & 4.00 \\
0.20 & 400 & 27 (114) & 460.60 & 0.11 & 196.27 & 3.83 \\
0.40 & 400 & 176 (429) & 1045.28 & 0.08 & 456.86 & 3.52 \\
0.60 & 400 & 578 (759) & 1394.61 & 0.10 & 508.52 & 3.55 \\
0.80 & 400 & 910 (673) & 1480.91 & 0.10 & 291.69 & 3.71 \\
\hline
\multicolumn{7}{c}{\textbf{CSIS with $\mC_3$}}\\
\hline
$\rho$ & $n$ & $\mathrm{MMMS}$ & $\mathrm{FP}_{\pi}$ & $\mathrm{FN}_{\pi}$ & $\mathrm{FP}_{\FDR}$ & $\mathrm{FN}_{\FDR}$\\
\hline
0.00 & 400 & 309 (919) & 100.00 & 0.89 & 14.83 & 2.69 \\
0.20 & 400 & 777 (1129) & 529.20 & 0.66 & 149.64 & 2.12 \\
0.40 & 400 & 1285 (1075) & 1087.79 & 0.56 & 333.27 & 1.96 \\
0.60 & 400 & 1572 (977) & 1383.80 & 0.58 & 336.54 & 2.06 \\
0.80 & 400 & 1629 (892) & 1485.02 & 0.57 & 178.37 & 2.79 \\
\hline
\end{tabular}
\end{center}
\end{table}

\subsection{Leukemia Data}

In this section, we demonstrate how CSIS can be used to do variable selection with an empirical dataset. We consider the leukemia dataset which was first studied by \cite{Golub} and is available at http://www.broad.mit.edu/cgi-bin/cancer/datasets.cgi. The data come from a study of gene expression in two types of acute leukemias, acute lymphoblastic leukemia (ALL) and acute myeloid leukemia (AML). Gene expression levels were measured using Affymetrix oligonucleotide arrays containing 7129 genes and 72 samples coming from two classes, namely 47 in class ALL  and 25 in class AML. Among these 72 samples, 38 (27  ALL and 11  AML) are set to be training samples and 34 (20  ALL and 14 AML) are set as test samples. For this dataset we want to select the relevant genes, and based on the selected genes estimate whether the patient has ALL or AML. AML progresses very fast and has a poor prognosis. Therefore, a consistent classification method that relies on gene expression levels would be very beneficial for the diagnosis.

In order to choose the conditioning genes, we take a pair of genes described in \cite{Golub} that result in low test errors. First is Zyxin and the second one is Transcriptional activator hSNF2b. Both genes have empirically high correlations for the difference between people with AML and ALL.

After conditioning on the aforementioned genes, we implement our conditional selection procedure using logistic regression. Using the random decoupling method, we select a single gene, TCRD (T-cell receptor delta locus). Although this gene has not been discovered by the ALL/AML studies so far, it is known to have a relation with T-Cell ALL, a subgroup of ALL (\citealt{TCell}). By using only these three genes, we are able to obtain a training error of 0 out of 38, and a test error of 1 out of 34. Similar studies in the past using sparse linear discriminant analysis or nearest shrunken centroids methods have obtained test errors of 1 by using more than 10 variables. We conjecture that this is due to the high correlation between the Zyxin gene and others, and that this correlation masks the information contained in the TCRD gene.

\subsection{Financial Data}
In this section we illustrate the advantages of conditional sure independence screening on a factor model with financial data. From the website
{http://mba.tuck.dartmouth.edu/pages} {/faculty/ken.french/}
we obtain 30 portfolios formed with respect to their industries. The returns for each portfolio are denoted by $y^j$ (for $j=1,\dots\,30$). The Fama-French three-factor model suggests that these returns follow the following equation
\begin{equation}
y^j_i=b_1^j f^1_i +b_2^j f^2_i + b_3^j f^3_i + \varepsilon_i,
\end{equation}
where $f^1$ is the excess return of the proxy market portfolio (given by the difference of the one-month T-Bill yield and the  value weighted return of all stocks on NYSE, AMEX and NASDAQ), $f^2$ is the difference between the return of small and big companies (measured by the difference of returns of two portfolios, one with companies that have small market cap and one with companies with large market cap) and finally $f^3$ is the difference of return from value companies and growth companies. This model was first proposed by \cite{FF93} and has been extensively analyzed since then. Since this seminal work, many other factors have been considered. In our numerical example, we used screening with the permutation test to detect if other factors are necessary. Besides the three factors mentioned above, we consider the momentum factor as an additional factor. This gives us 4 factors that are conditioned upon in CSIS.  For each given industrial portfolio, we also consider the returns from the other 29 portfolios as potential prediction factors.

We use daily returns data from 1/3/2002 to 12/31/2007. For each portfolio (30 in total), we first consider the marginal screening without conditioning. On average, for each portfolio, marginal screening picks 25.3 among 29 other industrial portfolios as predictors. This is mainly due to correlations between the returns of different portfolios. We next consider conditional marginal screening, in which the three Fama-French factors and the momentum factor are conditioned upon. As expected, the number of the selected variables decreases significantly to an average of 4.8. That is, about 4.8 portfolios on average can still have some potential prediction power in presence of the aforementioned four major factors.  The marginal and conditional fits of the values are given in Figure \ref{fig factors}. The black parts indicate the variables which are not included.

It is seen from these results that, conditional screening is more advantageous compared to marginal screening if few of the factors are known to be important. Furthermore, when there is significant correlation between some of the factors, as shown in the introduction, marginal screening considers most of the factors as relevant. In almost all financial models, stock returns are correlated with the return of the market portfolio. Therefore, in variable selection for financial factor models with many variables, one should always consider the returns conditional on the main driving forces of the market.

\begin{figure}[ht]
\centering
\subfigure[$\left| \hat{\beta}^M \right|$ using marginal screening]{
\includegraphics[width=0.45\textwidth]{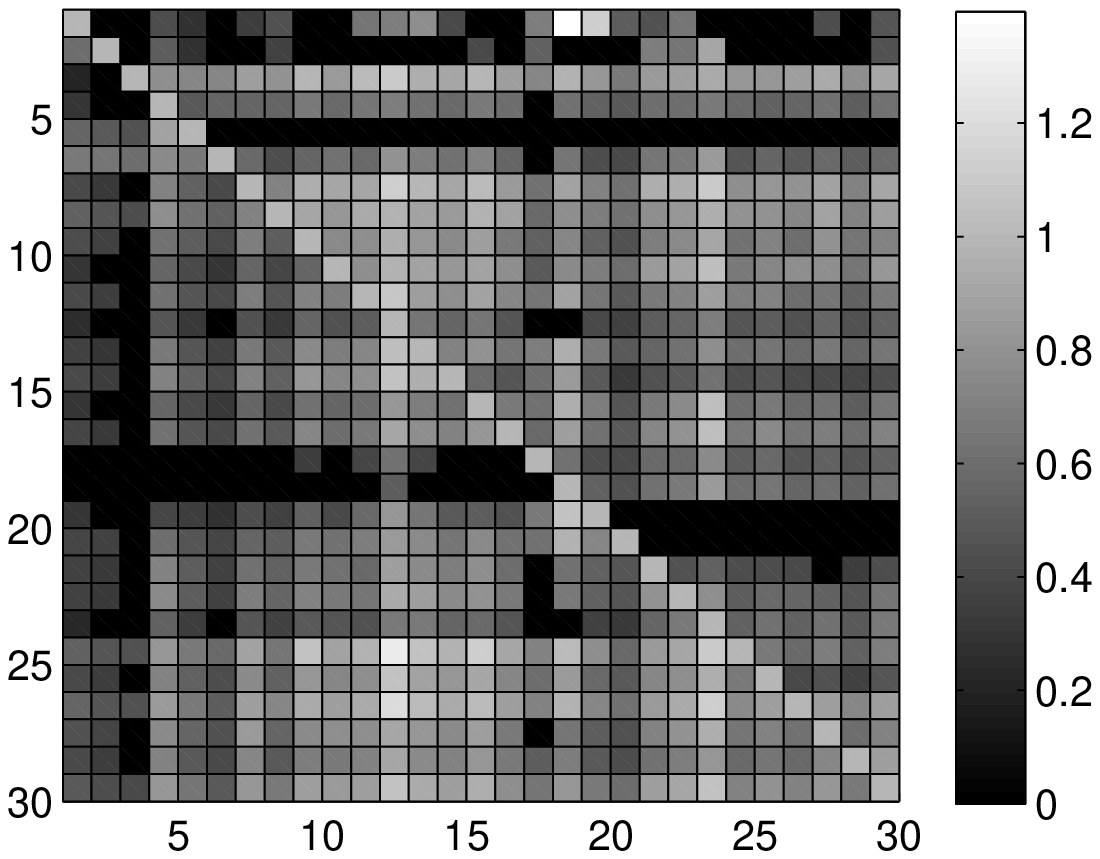}
}
\subfigure[$\left| \hat{\beta}^M \right|$ using conditional screening]{
\includegraphics[width=0.45\textwidth]{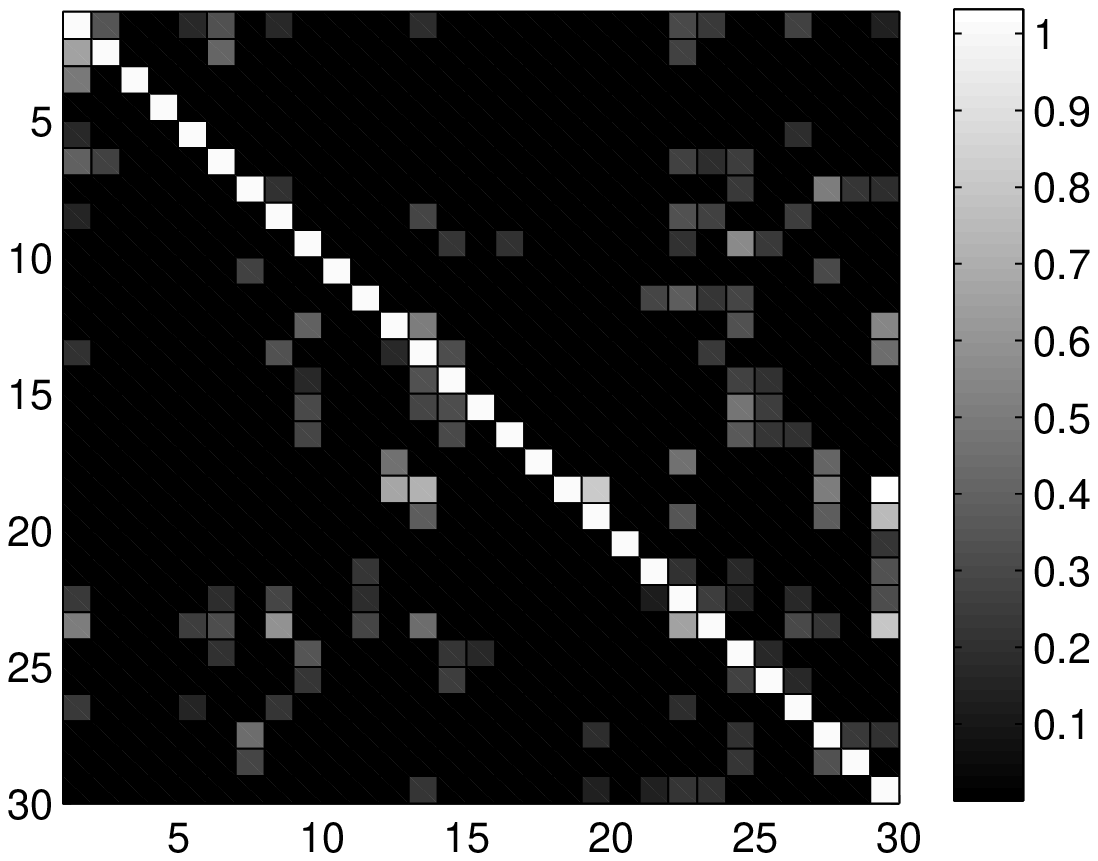}
}
\caption{Chosen factors with marginal (left) and conditional screening (right).}
\label{fig factors}
\end{figure}

\setcounter{figure}{0}
\section*{APPENDIX}

\renewcommand{\theequation}{A.\arabic{equation}}
\setcounter{equation}{0}

\renewcommand{\thesubsection}{A.\arabic{subsection}}
\setcounter{equation}{0} \setcounter{subsection}{0}

\subsection{Proof of Theorem \ref{thm1}}\label{app thm1}

\begin{proof}[Proof of Theorem \ref{thm1}]  The necessary part has already been proven in Section 3.1.  To prove the sufficient condition, we first note that condition
 $\Cov_L\left(Y,X_{j}\big|\mathbf{X_{\mathcal{C}}}\right)=0$ is equivalent to
\[
 \E b'(\bX_{\mC}^T \bbeta_{\mC}^M ) X_{j} = \E Y X_{j},
\]
as shown in Section 3.1.  This and (\ref{eq11})
imply that $( (\bbeta_{\mC}^M)^T, 0)^T$ is a solution to equation (\ref{eq9}).    By the uniqueness, it follows that $\bbeta_{\mC j}^M = ( (\bbeta_{\mC}^M)^T, 0)^T$, namely $\beta_j^M = 0$. This completes the proof.

\end{proof}

\subsection{Proof of Theorem \ref{thm2}}\label{app thm2}

\begin{proof}[Proof of Theorem \ref{thm2}]

We denote the matrix $\E m_j \bX_{\mC j}\bX_{\mC j}^T $ as $\Omega_j$ and partition it as
\[
\Omega_j=\left[
\begin{array}{cc}
\E m_j\bX_{\mC} \bX_\mC^T & \E m_j\bX_\mC \bX_j \\
\E m_j\bX_j \bX_\mC^T & \E m_j\bX_j^2 \end{array}
\right]=
\left[
\begin{array}{ccc}
\Omega_{\mC,\mC} & \Omega_{\mC,j}\\
\Omega_{\mC,j}^T & \Omega_{j,j}
\end{array}
\right].
\]
From the score equations, i.e. equations (\ref{eq9}) and  (\ref{eq11}), we have that
\begin{align*}
\E b' \lp \bX_{\mC}^T \bbeta_{\mC}^M \rp \bX_{\mC} = \E b' \lp \bX_{\mC j}^T \bbeta_{\mC j}^M \rp \bX_{\mC}.
\end{align*}
Using the definition of $m_j$, the above equation can be written as
\begin{align*}
\E m_j  (\bX_{\mC j}^T \bbeta_{\mC j}^M - \bX_{\mC}^T \bbeta_{\mC}^M) \bX_{\mC} = 0 .
\end{align*}
By letting $\bbeta_{\Delta,j}=\bbeta_{\mC j1}^M-\bbeta_{\mC}^M$, we have that
\[
    \E m_j  (\bX_{\mC}^T \bbeta_{\Delta, j}^M + X_{j}^T \bbeta_{j}^M) \bX_{\mC} = 0 .
\]
or equivalently
\begin{align} \label{eqA2}
\bbeta_{\Delta,j}=& - \Omega_{\mC,\mC}^{-1} \Omega_{\mC,j} \beta_j^M.
\end{align}
Furthermore, by (\ref{eq13}), we can express $\cov_L(Y,X_j|\bX_\mC)$ as
\begin{equation} \label{eq:prT2}
\cov_L(Y,X_j|\bX_\mC) = \E X_j \{ Y - \E_L (Y| \bX_{\mC}^T) \}.
\end{equation}
It follows from (\ref{eq12}) that
\begin{align} \label{eqA4}
\cov_L(Y,X_j|\bX_\mC)
 = \E X_j \left \{  b' \lp \bX_{\mC j}^T \bbeta_{\mC j}^M \rp - b' \lp \bX_{\mC}^T \bbeta_{\mC}^M \rp \right \}.
\end{align}
Using the definition of $m_j$ again, we have
\begin{align*}
\cov_L(Y,X_j|\bX_\mC) & = \E m_j X_j (\bX_{\mC j}^T \bbeta_{\mC j}^M - \bX_{\mC}^T \bbeta_{\mC}^M) \\
& = \E m_j X_j (\bX_{\mC}^T \bbeta_{\Delta, j}^M + X_{j}^T \bbeta_{j}^M)\\
&= \Omega_{\mC,j}^T  \bbeta_{\Delta,j} + \Omega_{j,j} \beta_j^M.
\end{align*}
By (\ref{eqA2}), we conclude that
\begin{equation} \label{eqA5}
 \cov_L(Y,X_j|\bX_\mC) = (\Omega_{j,j}-\Omega_{\mC,j}^T\Omega_{\mC,\mC}^{-1}\Omega_{\mC,j}) \beta_{j}^M.
\end{equation}

Now it is easy to see by Condition ~\ref{cond1} that
\[
|\beta_j^M| \geq c_2^{-1} |\cov_L(Y,X_j|\bX_\mC)| \geq c_3 n^{-\kappa},
\]
where $c_3 = c_1/c_2$.
Taking the minimum over all $j\in\mathcal{M}_{\mathcal{D}_\star}$ gives the result.
\end{proof}
\subsection{Proof of Theorem \ref{thm3}}\label{app thm3}

The proof of Theorem \ref{thm3} uses an exponential bound for a quasi maximum likelihood estimator. This bound is shown in \cite{SISGLM} and we repeat their theorem here to facilitate the reading.

Let $\bbeta_0=\arg \min_{\scriptsize \bbeta} \mathbb{E}l( \bX^T \bbeta,Y)$ the population parameter, which is an interior point of a large compact and convex set $\mathbf{B}\subset \R^p$.

\begin{condition}\label{ass Fisher int}
\qquad \vspace{0mm}
\begin{enumerate}
\item The Fisher information
\[
I\left(\bbeta\right)=\mathbb{E}\left\{ \left[\frac{\partial}{\partial\bbeta}l\left(\bX^{T}\bbeta,Y\right)\right]\left[\frac{\partial}{\partial\bbeta}l\left(\bX^{T}\bbeta,Y\right)\right]^{T}\right\} ,
\]
is finite and positive definite at $\bbeta=\bbeta_{0}$. Furthermore,
$\sup_{{\footnotesize \bbeta}\in\mathbf{B},\mathbf{x}}{\left\Vert I\left({ \bbeta}\right)^{1/2}\mathbf{x}\right\Vert }/{\left\Vert \mathbf{x}\right\Vert }$
exists.

\item The function $l(\bx^T\bbeta,y)$ is Lipschitz
with a positive constant $k_{n}$ for any $\bbeta$ in $\mathbf{B}$,
and $(\mathbf{x},y)$ in $\Lambda_n=\left\{ \mathbf{x},y:\left\Vert \mathbf{x}\right\Vert _{\infty}\leq K_{n}, |y|\leq K_{n}^{\star}\right\} $
with $K_{n}$ and $K_{n}^{\star}$ arbitrarily large constants.
Furthermore, there exists a constant $C$ such that
\begin{equation}\label{quasi tail cond}
\sup_{{\footnotesize \bbeta}\in\mathbf{B},\left\Vert {\footnotesize \bbeta} -{\footnotesize \bbeta_0}\right\Vert \leq Ck_{n}V_{n}^{-1}\left(p/n\right)^{1/2}}\left|\mathbb{E}\left[l\left(\bX^{T}\bbeta,Y\right)-l\left(\bX^{T}\bbeta_{0},Y\right)\right]\left(1-I_{n}\left(\bX,Y\right)\right)\right|\leq o\left(p/n\right),
\end{equation}
where $I_{n}\left(\bx,y\right)=I\left((\bx,y) \in \Lambda_n\right)$ with constant $V_n$ defined below.

\item The function $l\left(\bX^{T}\bbeta,Y\right)$
is convex in $\bbeta$ and
\[
\left|\mathbb{E}\left[l\left(\bX^{T}\bbeta,Y\right)-l\left(\bX^{T}\bbeta_{0},Y\right)\right]\right|\geq V_{n}\left\Vert \bbeta-\bbeta_{0}\right\Vert ^{2},
\]
for some positive constants $V_{n}$, and all $\left\Vert \bbeta-\bbeta_{0}\right\Vert\leq Ck_{n}V_{n}^{-1}\left(p/n\right)^{1/2}$.
\end{enumerate}
\end{condition}

\begin{thm} (\citealt{SISGLM}) \label{thm quasi}
Under Condition \ref{ass Fisher int}, for any $t>0$ it holds that
\[
\mathbb{P}\left(\sqrt{n}\left\Vert \hat{\bbeta}-\bbeta_{0}\right\Vert\geq16k_{n}\left(1+t\right)/V_{n}\right)\leq\exp\left(-2t^{2}/K_{n}^{2}\right)+n\mathbb{P}\left(\Lambda_{n}^c\right).
\]

\end{thm}

The proof of Theorem \ref{thm3} is based on Theorem \ref{thm quasi}.

\begin{proof}[Proof of Theorem \ref{thm3}]
By Lemma 1 of \cite{SISGLM}, Condition~\ref{cond2}(ii) gives the bound
\[
  P(|Y| \geq u) \leq s_1 \exp(-s_0 u).
\]
Hence, we have
\[
   \mathbb{P} ( \Lambda_n^c) \leq P(\|\bX\|_\infty > K_n) + P(|Y| \geq K_n^\star) \leq r_2 \exp(-r_0 K_n^\alpha).
\]
Using this and Theorem \ref{thm quasi}, letting $1+t=c_{3}V_n n^{1/2-\kappa}/\left(16k_{n}\right)$, we have
\begin{align*}
\mathbb{P}\left(\left|\hat{\beta}_{j}^{M}-\beta_{j}^{M}\right|\geq c_{3}n^{-\kappa}\right) & \leq\mathbb{P}\left(\left\Vert \hat{\bbeta}_{\mC j}^M-\bbeta_{\mC j}^M\right\Vert \geq c_{3}n^{-\kappa}\right)\\
 & \leq\exp\left(-c_4 n^{1-2\kappa}/\left(k_{n}K_{n}\right)^{2}\right)+nr_{2}\exp\left(-r_{0}K_{n}^{\alpha}\right),
\end{align*}
for some positive constant $c_4$.  Then, by Bonferroni's inequality, we obtain
\[
\mathbb{P}\left(\max_{q+1\leq j\leq p}\left|\hat{\beta}_{j}^{M}-\beta_{j}^{M}\right|\geq c_{3}n^{-\kappa}\right)\leq d\Big(\exp\big(-c_{4}n^{1-2\kappa}(k_{n}K_{n})^{-2}\big)+nr_{2}\exp\big(-r_{0}K_{n}^{\alpha}\big)\Big).
\]
This proves the first conclusion.

The second statement can be shown by considering the event
\[
\mathcal{A}_{n}=\left\{ \max_{j\in\mathcal{M}_{\star\mathcal{D}}}\left|\hat{\beta}_{j}^{M}-\beta_{j}^{M}\right|\leq c_{3}n^{-\kappa}/2\right\}.
\]
On the event $\mathcal{A}_{n}$, by Theorem \ref{thm2}, it holds that for all $j\in\mathcal{M}_{\star\mathcal{D}}$
\[
\left|\hat{\beta}_{j}^{M}\right|\geq c_{3}n^{-\kappa}/2.
\]
By letting $\gamma=c_{5}n^{-\kappa}\leq c_{3}n^{-\kappa}/2$, on the event $\mathcal{A}_n$ we have
the sure screening property, that is $\mathcal{M}_{\star\mathcal{D}}\subset\hat{\mathcal{M}}_{\mathcal{D},\gamma}$.
The probability bound can be shown by using the first result along
with Bonferroni's inequality over all chosen $j$, which gives
\[
\mathbb{P}\left(\mathcal{A}_{n}^c\right)\leq s\left[\exp\left(-c_{4}n^{1-2\kappa}(k_{n}K_{n})^{-2}\right )
+nr_{2}\exp\left(-r_{0}K_{n}^{\alpha}\right)\right].
\]
This completes the proof.

\end{proof}

\begin{comment}

\subsection{Proof of Proposition \ref{prop1}}\label{app prop}

\begin{proof}[Proof of Proposition \ref{prop1}]
First of all, note that when $\bX=\left[\bX_{\mC}, \bX_{\mD} \right]$ has a multivariate normal distribution, then $\E \left[ \bX_{\mD} | \bX_{\mC} \right] = \bSigma_{\mD,\mC} \bSigma_{\mC, \mC}^{-1} \bX_{\mC}$, where
\[
\bSigma=\left(
    \begin{array}{cc}
      \bSigma_{\mC,\mC} & \bSigma_{\mC,\mD} \\
      \bSigma_{\mD,\mC} & \bSigma_{\mD,\mD} \\
    \end{array}
  \right)
.
\]
Plugging this in the definition of $\bZ$,
\begin{align}
\bZ & =  \E \Big[ \E \big[\bX_{\mD} | \bX_{\mC}\big] \E \big[\bX^T \bbeta^\star -\bX_{\mC}^T \bbeta_{\mC}^M | \bX_{\mC}\big] \Big]\nonumber\\
&=\E  \bSigma_{\mD,\mC} \bSigma_{\mC,\mC}^{-1} \bX_{\mC} \big(\E \big[\bX^T \bbeta^\star -\bX_{\mC}^T \bbeta_{\mC}^M | \bX_{\mC}\big] \big).
\label{eq:prop1a}
\end{align}
Similar calculations as in the proof of Theorem \ref{thm1} give that
\begin{equation}\label{eq:prop1b}
\bbeta_{\mC}^\star - \bbeta_{\mC}^M = - \bSigma_{\mC,\mC}^{-1} \bSigma_{\mC, \mD} \bbeta_{\mD}^\star.
\end{equation}
Combing equations (\ref{eq:prop1a}) and (\ref{eq:prop1b}) and taking the expectation yields that $\bZ=(0,\ldots,0)^T$.
\end{proof}
\end{comment}

\subsection{Proof of Theorem \ref{thm4}}\label{app thm4}

\begin{proof}[Proof of Theorem \ref{thm4}] The first part of the proof  is  similar  to that of Theorem 5 of \cite{SISGLM}. The idea of this proof is to show that
\begin{equation}
\|\bbeta_{\mD}\|^{2} = O \lp \lambda_{\max}\lp \bSigma_{\mD|\mC} \rp \rp.\label{eq norm betaM}
\end{equation}
If this holds, the size of the set $\{j=q+1,\ldots,p:|\beta_{j}^{M}|>\varepsilon n^{-\kappa}\}$
can not exceed $O\lp n^{2\kappa}\lambda_{\max}\lp \bSigma_{\mD|\mC} \rp \rp$
for any $\varepsilon>0$. Thus on the event
\[
\mathcal{B}_{n}=\left\{ \max_{q+1\leq j\leq p}|\hat{\beta}_{j}^{M}-\beta_{j}^{M}|\leq\varepsilon n^{-\kappa}\right\} ,
\]
the set $\{j=q+1,\ldots,p:|\hat{\beta}_{j}^{M}|>2\varepsilon n^{-\kappa}\}$
is a subset of the set $\{j=q+1,\ldots,p:|\beta_{j}^{M}|>\varepsilon n^{-\kappa}\}$,
whose size is bounded by $O \lp n^{2\kappa}\lambda_{\max}\lp \bSigma_{\mD|\mC} \rp \rp$.
If we take $\varepsilon=c_{5}/2$, we obtain that
\[
\mathbb{P}\lp |\hat{\mathcal{M}}_{\mathcal{D},\gamma}|\leq O \lp n^{2\kappa}\lambda_{\max}\lp \bSigma_{\mD|\mC} \rp \rp \rp \geq \mathbb{P} (\mathcal{B}_{n}).
\]
 Finally, by Theorem \ref{thm3}, we obtain that
\[
\mathbb{P}(\mathcal{B}_{n})\geq1-d\Big(\exp\big(-c_{4}n^{1-2\kappa}(k_{n}K_{n})^{-2}\big)
    +nr_{2}\exp\big(-r_{0}K_{n}^{\alpha}\big)\Big)
\]
 and therefore the statement of the theorem follows.

We now prove \eqref{eq norm betaM}
by using $\Var(\bX^{T}\bbeta^{\star})=O(1)$ and (\ref{eqA5}).  By Condition~\ref{cond3}(ii), the Schur's complement $(\Omega_{j,j}-\Omega_{\mC,j}^T\Omega_{\mC,\mC}^{-1}\Omega_{\mC,j})$ is uniformly bounded from below. Therefore, by (\ref{eqA5}), we have
\[
     |\beta_{j}^M| \leq D_1 | \cov_L(Y,X_j|\bX_\mC)|,
\]
for a positive constant $D_1$.  Hence, we need only to bound the conditional covariance.

By (\ref{eqA4}), (\ref{eq9}) and  Lipschitz continuity of $b'(\cdot)$, we have
\begin{eqnarray*}
| \cov_L(Y,X_j|\bX_\mC) |
 & = &\E \bigl |X_j \left \{  b' \lp \bX^T \bbeta^* \rp - b' \lp \bX_{\mC}^T \bbeta_{\mC}^M \rp \right \} \bigr  | \\
&\leq& D_{2} \E \bigl |X_{j}(\bX^{T}\bbeta^{\star}-\bX_{\mC}^{T}\bbeta_{\mC}^M) \bigr |\\
&=&D_{2}\E \bigl |X_{j}[\bX_\mC^{T} \bbeta_\mC^\Delta +\bX_\mD^{T}\bbeta_\mD^{\star}] \bigr |.
\end{eqnarray*}
where $\bbeta_\mC ^\Delta= (\bbeta_\mC^{\star}-\bbeta_{\mC}^M)$.  Writing the last term in the vector form, we need to bound
\begin{eqnarray*}
 \| \E \bX_{\mD}\bX_{\mD}^{T}\bbeta_{\mD}^{\star} + \bX_{\mD}\bX_{\mC}^{T}\bbeta_{\mC}^\Delta \|^2.
\end{eqnarray*}
From the property of the least-squares, we have
$\E [ \E_L (\bX_{\mD} |  \bX_{\mC})  \bX_{\mC}^T ] = \E [ \bX_{\mD}   \bX_{\mC}^T ]$.  Thus the above expression can be written as
\begin{eqnarray*}
 \|   \big[\bSigma_{\mD|\mC} \big] \bbeta_{\mD}^{\star} + \E \E_L (\bX_{\mD} |  \bX_{\mC})  [\bX_{\mC}^{T}\bbeta_{\mC}^\Delta + \E_L(\bX_{\mD}^{T}|\bX_{\mC}) \bbeta_{\mD}^*)] \|
& = & \left\|  \big[\bSigma_{\mD|\mC} \big] \bbeta_{\mD}^\star + \bZ \right\|^2,
\end{eqnarray*}
recalling the definition of $\bZ =  \E \E_L (\bX_{\mD} |  \bX_{\mC})  \lp  \bX^T \bbeta^\star -\bX_{\mC}^T \bbeta_{\mC}^M \rp$ in Condition \ref{cond3}.

Using the law of total variance, we have that
\begin{eqnarray*}
\left\|  \big[\bSigma_{\mD|\mC} \big] \bbeta_{\mD}^\star + \bZ \right\|^2&=&{\bbeta_{\mD}^\star}^T  \big[\bSigma_{\mD|\mC} \big]^2\bbeta_{\mD}^\star + 2 \bZ^T  \big[\bSigma_{\mD|\mC} \big] + \bZ^T \bZ\\
&\leq&\lambda_{\max} \lp \big[\bSigma_{\mD|\mC} \big] \rp \lp {\bbeta_{\mD}^\star}^T  \big[\bSigma_{\mD|\mC} \big]\bbeta_{\mD}^\star\rp + 2 \bZ^T  \big[\bSigma_{\mD|\mC} \big] + \bZ^T \bZ \\
&\leq&\lambda_{\max} \lp \big[\bSigma_{\mD|\mC} \big] \rp \Var(\bX^T \bbeta^\star)+ 2 \bZ^T  \big[\bSigma_{\mD|\mC} \big] + \bZ^T \bZ,
\end{eqnarray*}
and the last two terms are $o\lp \lambda_{\max} \lp  \big[\bSigma_{\mD|\mC} \big] \rp \rp$ due to Condition \ref{cond3}. Therefore, we have that
\begin{equation*}
\|\bbeta_{\mD}\|^{2}=O \lp \lambda_{\max}\lp \big[\bSigma_{\mD|\mC} \big] \rp \rp,
\end{equation*}
and that gives us the desired result.
\end{proof}

\subsection{Proof of Theorem \ref{thm FDR}}\label{app thm FDR}

\begin{proof}[Proof of Theorem \ref{thm FDR}]
Note that the false discovery proportion can be rewritten as
\[
\mathbb{E}\left(\frac{\left|\hat{\mathcal{M}}_{\mathcal{D},\delta}\cap(\mathcal{M}_{\star\mathcal{D}})^c\right|}{\left|(\mathcal{M}_{\star\mathcal{D}})^c\right|}\right)=\frac{1}{d-\left|\mathcal{M}_{\star\mathcal{D}}\right|}\sum_{j\in(\mathcal{M}_{\star\mathcal{D}})^c}\mathbb{P}\left(I_{j}\left(\hat{\beta}_{j}^{M}\right)^{1/2}\left|\hat{\beta}_{j}^{M}\right|\geq\delta\right).
\]

With the given conditions, by Theorem 1, we have $\beta_j^M = 0$.  Since $\bX_{\mC}$ includes the intercept term, $\E e_i = 0$.  It is known that $I_{j}\left(\hat{\beta}_{j}^{M}\right)^{1/2}\left|\hat{\beta}_{j}^{M}\right|$ (for $j \in (\mathcal{M}_{\star\mathcal{D}})^c$)
has an asymptotically  standard normal distribution (Gao \Etal, 2008,
Heyde, 1997). Then, it follows that for a $c_7>0$
\[
\sup_{z}\left|\mathbb{P}\left(I_{j}\left(\hat{\beta}_{j}^{M}\right)^{1/2}\left|\hat{\beta}_{j}^{M}\right|\geq z\right)-\Phi(z)\right|\leq c_7 n^{-1/2}.
\]

Combining both equations, we obtain
\[
\mathbb{E}\left(\frac{\left|\hat{\mathcal{M}}_{\mathcal{D},\delta}\cap(\mathcal{M}_{\star\mathcal{D}})^c\right|}{\left|(\mathcal{M}_{\star\mathcal{D}})^c\right|}\right)\leq\frac{1}{d-\left|\mathcal{M}_{\star\mathcal{D}}\right|}\sum_{j\in(\mathcal{M}_{\star\mathcal{D}})^c}\left(2\left(1-\Phi\left(\delta\right)\right)+c_7 n^{-1/2}\right).
\]

Setting $\delta=\Phi^{-1}\left(1-\frac{f}{2d}\right)$ gives the result.
\end{proof}

\end{document}